\documentclass[12pt,reqno]{amsart}
  
\newcommand{\cs}{$\clubsuit$} % THIS IS TO MARK COMMENTS, CHANGES
\usepackage{latexsym}
\usepackage{amssymb}

\renewcommand{\Im}{{\operatorname{Im}\,}}

\renewcommand{\epsilon}{\varepsilon}

\newcommand{\N}{{\mathbb N}}
\newcommand{\R}{{\mathbb R}}

\newcommand{\Z}{{\mathbb Z}}

\newcommand{\h}{\hbar}

\renewcommand{\phi}{\varphi}

\newcommand{\bcal}{\mathcal{B}}

\newcommand{\ical}{\mathcal{I}}

\newcommand{\lcal}{\mathcal{L}}

\newcommand{\scal}{\mathcal{S}}

\newtheorem{theo}{{\sc Theorem}}[section]

\newtheorem{cor}[theo]{{\sc Corollary}}

\newtheorem{lem}[theo]{{\sc Lemma}}

\newtheorem{prop}[theo]{{\sc Proposition}}

\newenvironment{rem}{\medskip\noindent{\it Remark:\/} }{\medskip}

\newtheorem{defn}[theo]{{\sc Definition}}

\title[Eigenfunction bounds]{Averaged pointwise bounds for deformations of Schr\"{o}dinger eigenfunctions}

\author{Suresh Eswarathasan and John A. Toth}
\address{Department of Mathematics and Statistics, McGill University, Montr\'eal, Canada}
\email{jtoth@math.mcgill.ca} 
\address{Centre de Recherches Mathematiques, Universit\'e de Montr\'eal, Montr\'eal, Canada} 
\address{Department of Mathematics and Statistics, McGill University, Montr\'eal, Canada}
\email{suresh@math.mcgill.ca \\ eswarath@crm.umontreal.ca}

\date{}

\begin{document}

\maketitle
\begin{abstract}  Let $(M,g)$ be an $n$-dimensional, compact Riemannian manifold and $P_0(\h) = -\hbar^2 \Delta_g + V(x)$ be a semiclassical Schr\"odinger operator with $\hbar \in (0,\hbar_0]$.  Let $E(\h) \in  [E-o(1),E+o(1)]$  and $(\phi_{\h})_{\h \in (0,\h_0]}$ be a family of $L^2$-normalized eigenfunctions of  $P_0(\h)$ with $P_0(\h) \phi_{\h} = E(\h) \phi_{\h}.$ We consider magnetic deformations of $P_0(\h)$ of the form $P_u(\h) = - \Delta_{\omega_u}(\hbar) + V(x)$, where $\Delta_{\omega_u}(\hbar) = (\hbar d + i \omega_u(x))^* (\hbar d + i \omega_u(x)).$ Here, $u$ is a $k$-dimensional parameter running over $B^k(\epsilon)$ (the ball of radius $\epsilon$), and the family of the magnetic potentials $(w_u)_{u\in B^k(\epsilon)}$ satisfies the admissibility condition given in Definition 1.1. This condition implies that $k\ge n$ and is generic under this assumption.
 
 Consider the corresponding family of deformations of $(\phi_{\h})_{\h \in (0, \h_0]}$, given by $(\phi^u_{\h})_{\h \in (0, \h_0]}$, where 
$$ \phi_{\h}^{(u)}:= e^{-it_0 P_u(\h)/\h} \phi_{\h}$$
for $|t_0| \in (0, \epsilon)$; the latter functions are themselves eigenfunctions of the $\h$-elliptic operators
$Q_u(\h): = e^{-it_0P_u(\h)/\h} P_0(\h) e^{it_0 P_u(\h)/\h}$ with eigenvalue $E(\h)$ and $Q_0(\h) = P_{0}(\h).$  Our main result, Theorem \ref{supbound}, states that for $\epsilon >0$ small, there are  constants $C_j=C_j(M,V,\omega,\epsilon)>0$ with $j=1,2$ such that 
 $$  C_{1} \leq  \int_{\bcal^k(\epsilon)} |\phi_{\h}^{(u)}(x)|^2 \, du \leq C_{2},$$
 uniformly for $x \in M$ and $\h \in (0,h_0].$  We also give an application to eigenfunction restriction bounds in Theorem \ref{restriction}.
%Given $\chi \in C^{\infty}_{0}(\R;[0,1])$ a standard cutoff function supported in a small ball $B(\delta) \subset M$,  we consider $n$-parameter family of locally "bumped" metrics $g_u = e^{\psi_u} \chi g_0 | ; u \in \ical_k(\delta)$ of $g_0.$ with the property that the variation $\delta \psi: M \rightarrow {\mathbb R}^{k}$ is locally (ie. $\partial_{x_i} \delta \psi (i)(x) \neq 0; x \in B(\delta).$. 

\end{abstract}

\section{Introduction} 
\subsection{Background} %As in the abstract, consider a family of conformal metrics $g_u = e^{\psi_u} g_0; u \in \bcal(\delta)$ with associated Laplace-Beltrami operators $\Delta_u: C^{\infty}(M) \rightarrow C^{\infty}(M)$ and $L^2(M,g_u)$-normalized eigenfunctions $\phi_{j(u)}; j=1,2,....$ For the unperturbed metric $g_0$, for simplicity we just write the eigenvalues $\lambda_k:= \lambda_k(0)$ and similarly for the eigenfunctions. 
Let $(M,g)$ be an $n$-dimensional, compact Riemannian manifold with Laplace Beltrami operator $-\Delta_{g}: C^{\infty}(M) \rightarrow C^{\infty}(M)$. We let $ 0 = \lambda_1 \leq \lambda_2 \leq ...$ be the eigenvalues of $- \Delta_g$ and $\{\phi_{\lambda_j}\}_{j \in Z^+}$ be the corresponding $L^{2}$ normalized basis eigenfunctions.  The relationship between the large $\lambda$ asymptotics of the $\phi_{\lambda}$'s and the corresponding dynamics of the geodesic flow $G^{t}: T^*M \rightarrow T^*M$ is a very interesting subject which has been studied a great deal over the past several decades. Nevertheless, many basic questions related to this correspondence are not well understood, especially for eigenfunctions; there are several ways of quantifying this correspondence. One such approach is to classify the possible semiclassical defect measures: Given a pseudodifferential operator of order zero, known in physics as the quantum observable, one studies the possible limit measures $d\mu$ on $S^*M$ satisfying $ \lim_{\lambda \rightarrow \infty} \langle A \phi_{\lambda}, \phi_{\lambda} \rangle = \int_{S^*M} \sigma(A) d\mu.$   Recently, there have been important advances in the understanding of defect measures vis-a-vis quantum unique ergodicity \cite{A,H,L}.  

One important measure of the classical-quantum correspondence, different from that of defect measures, involves computation of $L^{\infty}$ bounds, or more generally $L^{p}$ bounds, of eigenfunctions. The universal estimate $\| \phi_{\lambda} \|_{L^{\infty}} = {\mathcal O}(\lambda^{\frac{n-1}{2}})$ due to Avakumovic, H\"ormander, and Levitan \cite{Av, Ho, Lev} follows from the well-known pointwise asymptotic formula
 $$ e(x,x,\lambda) = (2\pi )^{-n} \text{vol}(M)  \lambda^{n} + {\mathcal O}(\lambda^{n-1}) $$ 
 which is uniform in $x \in M$. Here, $e(x,y,\lambda):= \sum_{\lambda_j \leq \lambda} \phi_{\lambda_j}(x) \overline{\phi_{\lambda_j}(y)}$ is the Schwartz kernel of the spectral projector of $-\Delta_g$. The ${\mathcal O}(\lambda^{\frac{n-1}{2}})$ bound is known to be saturated by the zonal harmonics on ${\mathbb S}^{n}$ at the north and south poles with the round metric. Sogge and Zelditch \cite{SZ} proved that  this maximal bound is only attained at points $z \in M$, where $|\lcal_{z}| >0$ where $\lcal_{z}:= \{ \xi \in S_{z}^*M;   \exists \,  T>0  \, \text{such that}  \, \exp_{z} (T \xi) = z \}$ and $|A|$ denotes the Liouville measure of $A \subset S_{z}^{*}M.$ The latter  theorem  is closely related to earlier work of Safarov \cite{Sa} and is the pointwise analogue of a result of D\"uistermaat and Guillemin \cite{DG} and Ivrii \cite{I} for the Weyl remainder 
 of the integrated spectral function, $N(\lambda) = \int_{M} e(x,x;\lambda) d vol(x).$ More recently, Sogge, Toth, and Zelditch \cite{STZ} improved these general $L^{\infty}$ estimates by replacing the  zero-measure looping condition in \cite{SZ} with the weaker assumption of  recurrence.
 %  sets: Given $x \in M$, one defines the recurrence set
 % $$ \rcal_{x} = \{ \xi \in \lcal_x; \xi \in \omega(\xi) \},$$
 % where $\omega$ denotes the $\omega$-limit of the vector $\xi$ under iterates of the first return map. In \cite{STZ} it is proved that
 % $$ |\rcal_x| = 0 \,\, \forall x \in M \, \, \Rightarrow \,\, \| \phi_{\lambda} \|_{L^{\infty}}  = o(\lambda^{\frac{n-1}{2} }).$$
 %The simple example of a triaxial ellipsoid shows that the latter sup bound is a substantial improvement over the full loopset result. Indeed, when $x \in M$ is an umbilic point, $|\lcal_x| =1$ but $|R_{x}| = 0$ and so, in general, the recurrence set can be much smaller than the full loopset.
 
 It is difficult to improve the general ${\mathcal O}(\lambda^{\frac{n-1}{2}})$ by  polynomial powers of $\lambda$ and there are few rigorous results along these lines; see \cite{Z} for a thorough overview of known results on eigenfunction bounds.  In the quantum completely integrable case, explicit  polynomial improvements in the $L^{\infty}$ bounds for eigenfunctions  are given in \cite{T}.   % In the case of manifolds with variable negative curvature,  the best-known sup bounds are ${\mathcal O}(\lambda^{\frac{n-1}{2}} /\log \lambda)$ \cite{Be}.  
For arithmetic hyperbolic surfaces, polynomial improvements were obtained by Iwaniec and Sarnak \cite{IS}.  A conjecture of Sarnak \cite{S} asserts that for hyperbolic surfaces, $\| \phi_{\lambda} \|_{L^{\infty}} = {\mathcal O}_{\epsilon}(\lambda^{\epsilon})$ for any $\epsilon >0$. Both numerical evidence and  Berry's random wave model \cite{Ber} are consistent with such a bound, but rigorous results at this time remain elusive; see the lecture notes of Sarnak \cite{S} for further details on this conjecture and related topics.

%Indeed, an application of the law of the iterated logarithm in the random wave model \cite{ABST, S} gives the eigenfunction supremum bound predicition
%\begin{equation} \label{rwbound}
%\| \phi_{\lambda} \|_{L^{\infty}} = {\mathcal O}(\sqrt{ \log \lambda} ). \end{equation}

%Such a bound is consistent with the normal amplitude distribution function for random waves and is closely related to classical results on $L^{\infty}$-bounds for random Fourier series \cite{Sa-Zy} and their generalizations \cite{K}. The only rigorous results that we are aware of for actual eigenfunction sup bounds which come close to the random wave prediction in (\ref{rwbound}) are due to 
 % Vanderkam \cite{V} and are related to earlier work of Bourgain \cite{B} on $L^{\infty}$-bounds for eigenfunctions in the ball ${\mathbb B}^{n} \subset {\mathbb R}^{n+1}.$ 
  
  In the more general semiclassical setting, Koch, Tataru and Zworski \cite{KTZ} proved sharp $L^p$ eigenfunction estimates for general semiclassical Schr\"odinger operators $P(\h) = -\h^2 \Delta_g + V(x)$, where $\h \in (0,\h_0]$ is the semiclassical parameter. In particular, they prove that
  $\| \phi_\h \|_{L^\infty} = {\mathcal O}(\h^{ \frac{-n-1}{2} })$ extending the homogeneous Avakumovic-H\"{o}rmander-Levitan bound.
  
  One of the main applications of the results in this paper is to eigenfunction restriction bounds; in this case, more is known. Let $H \subset M$ be denote a smooth hypersurface. In the general homogeneous case, Burq, Gerard and Tzvetkov \cite{BGT} gave sharp general upper bounds for $\| \phi_\lambda\|_{L^p(H)}$ which depend on geodesic curvature of $H$; in particular, when $H$ is totally geodesic, the authors show that  $\| \phi_\lambda\|_{L^2(H)} = {\mathcal O}(\lambda^{\frac{n-1}{4}})$ and if $H$ has strictly positive geodesic curvature this bound improves to $\| \phi_\lambda\|_{L^2(H)} = {\mathcal O}(\lambda^{\frac{n-1}{6}}).$  Similar results were obtained independently using the analysis of fold and blowdown singularities of Fourier integral operators by Hu \cite{Hu}.  Recently, Hassell and Tacy generalized these results to the semiclassical setting \cite{HT}.
   
   The article \cite{T2} proved that for generic curves $H$ in the quantum completely integrable case on surfaces, there is a large improvement in the upper bound, explicitly of the form $\| \phi_\lambda\|_{L^2(H)} = {\mathcal O}( \sqrt{\log \lambda})$, which is consistent with the random wave model.  Here, ``genericity" is defined in terms of the associated moment mapping of the integrable system.  At the opposite extreme in \cite{TZ1,TZ2}, it is proved that for generic hypersurfaces $H$ of manifolds with ergodic geodesic flow, along with quantum ergodic sequences of eigenfunctions, one actually has {\em quantum ergodic restriction} in the sense that the asymptotics take the form $\| \phi_\lambda\|_{L^2(H)} \sim 1.$  We show in Theorem \ref{restriction} that for any submanifold $H \subset M$  there is a natural class of unitary perturbations of Schr\"odinger eigenfunctions that satisfy the $\asymp 1$ restriction bound, at least for a positive measure of values of the perturbation parameters.   This is consistent with the ergodic case and our results here can be interpreted as further evidence for the efficacy of the random wave model. 
   
\subsection{Discussion of Results}
 We now describe our results in more detail.  Let $\bcal^{k}(\epsilon) \subset \R^k$ be the $k$-ball of radius $\epsilon >0$ centered at $0 \in \R^k.$ In the following discussion, $(\omega_u)_{u \in \bcal^k(\epsilon)}$  denotes  a smooth $k$-parameter family of $C^{\infty}$ one-forms on $M$. In each local coordinate chart $U \subset \R^n,$ one can write
 \begin{equation} \label{magform}
 \omega_{u}(x)= \sum_{j=1}^{n} \omega_j(x,u) dx_j \text{ for} \,\,\, \omega_j \in C^{\infty}(U \times \bcal^k(\epsilon)), \,\,j=1,...,n. \end{equation}
 Consider the associated gauge-invariant semiclassical differentials 
  $$d_{\omega_u}(\h):=  \h d + i \omega_u: C^{\infty}(M) \rightarrow \Omega^1(M).$$ We assume throughtout that $\omega_0 = 0$, $\h \in (0,\h_0]$ is the semiclassical parameter, and $d_{\omega_u}(\h)^*: \Omega^1(M) \rightarrow C^{\infty}(M)$ is the adjoint induced by the Riemannian metric $g$.  Explicitly,  $d_{\omega_u}(\h)^* = \h d^* +  *\omega_u$, where $*: \Omega^{1}(M) \rightarrow \Omega^{n-1}(M)$ is the Hodge star operator and for $w \in \Omega^1(M),$ the {\em{codifferential}} is given by 
$$ d^{*}w = - * d *w = \frac{1}{\sqrt{|g|}} \sum_{i,j} \partial_{x_j} [ \sqrt{|g|} g^{ij} w_i].$$
Then, for any given electric potential $V \in C^{\infty}(M)$, we form a family of semiclassical magnetic Schr\"odinger operators 
 \begin{equation} \label{magnetic}
  P_{u}(\h):=  - d_{\omega_u}^*(\h) d_{\omega_u}(\h)  + V(x). \end{equation}
For $u=0$, we clearly have that $P_0(\h) = -\h^2 \Delta_g + V(x),$ where $-\Delta_g: C^{\infty}(M) \rightarrow C^{\infty}(M)$ is the standard Laplacian. 
The operator $P_u(\h)$ has the principal symbol
\begin{equation} \label{magsymbol}
p_{u}(x,\xi) = \sum_{i,j =1}^n (\xi_i + \omega_i(x,u)) g^{ij}(x) (\xi_j + \omega_j(x,u)) + V(x), \end{equation}
where the $\omega_j$ are the local components of the magnetic potentials in (\ref{magform}).
The bicharacteristic flow associated with the Hamiltonian vector field $H_{p_{u}}$ is denoted by
\begin{equation} \label{bichar}
 G^{t}_{u}: T^*M \rightarrow T^*M, \,\, t \in \R. \end{equation}
 
   Theorem \ref{supbound}, the main result of this article, gives an explicit ansatz for mollifying the pointwise behavior of eigenfunctions, on average, by using propagators $e^{it_0 P_u(\h)/\h} $ generated by the magnetic Schr\"odinger operators
  (\ref{magnetic}).
  
  %  Here, we use the notation 
% in (\ref{magnetic}) for the operator $P^{(u)}(h) = -h^2 \nabla_u^* \nabla_u + V: C^{\infty}(M,L) \rightarrow C^{\infty}(M,L)$ acting on the line bundle $L$ with curvature $\omega$ and associated connection $\nabla_u: C^{\infty}(M,L) \rightarrow C^{\infty}(M,L).$ Our analysis is entirely local in $x \in M,$ so we will just  view the $P^{(u)}(h)$ as operating on functions. We now describe our results in more detail.
 
  Let $(M^n,g)$ be a compact manifold and $P_0(\h) = -\hbar^2 \Delta_g + V(x)$ be a semiclassical Schr\"odinger operator with $\hbar \in (0,\h_0]$ with $E \in \R$ a regular value of $p_0(x,\xi) = |\xi|_g^2 + V(x);$ here, $E(\h) \in  [E-o(1),E+o(1)]$  and $(\phi_{\h})_{\h \in [0, \h_0)}$, where $\phi_\h \in C^{\infty}(M)$, is an $L^2$-normalized family of semiclassical eigenfunctions with $P_0(\h) \phi_\h = E(\h) \phi_{\h}.$ We consider magnetic deformations of $P_0(\h)$ of the form (\ref{magnetic}) and assume that the number of deformation parameters $k \geq n = \dim M.$ In terms of local coordinates $x=(x_1,...x_n)$,  we make the following 
 \begin{defn} \label{admissible}
 We say that the smooth family of magnetic potentials $ (\omega_u)_{u \in \bcal^k(\epsilon)}$ with $\omega_u(x) = \sum_{j=1}^{n} \omega_j(x,u) dx_j \in \Omega^1(M; C^{\infty}(\bcal^k(\epsilon) )$  (interchangably, the corresponding operator $P_u(\h)$)  is \underline{\em admissible}  provided
 \begin{itemize}
  \item $\omega_0(x) = 0\,$ for all $x \in M,$
  \item  The family of maps $ f_x: \bcal^k(\epsilon) \rightarrow \R^n$  with $k \geq n,$ given by the components 
  $$ f_{x}(u) = ( \omega_1(x,u),...,\omega_n(x,u) ),$$
are submersions for all $x \in M.$ 
  \end{itemize} 
  \end{defn}
  The  submersion requirement in Definition (\ref{admissible}) is clearly a coordinate-independent and generic condition \cite{Mo}. In any local coordinate chart and for $\epsilon >0$ small, this amounts to verifying that 
 the first variation $$\delta \omega(x) := \partial_{u_j} \omega_{i}(x,u) |_{u=0},$$ for $i=1,...,n$ and $j=1,...,k$, is a $n \times k$ matrix of rank  $n.$
  For $|t_0|$ small, we define the corresponding deformations of the eigenfunctions  by
 \begin{equation} \label{deform1}
 \phi_\h^{(u)}:= e^{-it_0 P_u(\h)/\h} \phi_\h. \end{equation}
 Consider the initial-value problem 

\begin{equation}\begin{cases} \label{schroedinger}
 \h D_{t} \Phi_\h^{(u)}(t) + P_u(\h) \Phi_\h^{(u)}(t) = 0, \\
 \Phi_\h^{(u)}(0) = \phi_\h.
\end{cases}
\end{equation}  

\noindent The deformed eigenfunctions (\ref{deform1}) are non-stationary solutions of  the time-dependent Schr\"odinger equation (\ref{schroedinger}) evaluated at time $t=t_0;$ that is,
 $$ \Phi_\h^{(u)}(t_0) = \phi_\h^{(u)}.$$
It is also useful to note that when viewed as stationary functions, the $\phi_\h^{(u)}$ are also $L^2$-normalized eigenfunctions of the $\h$-elliptic operators
\begin{equation} \label{helliptic}
Q_u(\h): = e^{-it_0P_u(\h)/\h} P_0(\h) e^{it_0 P_u(\h)/\h} \end{equation}
with eigenvalue $E(\h).$  The $\phi_\h^{(u)}$ are clearly deformations, up to constant multiples of modulus 1, of the eigenfunction $\phi_\h$ with $\phi_\h^{(0)} = e^{-i t_0 E(\h)/\h}  \phi_\h$ and $Q_0(\h) = P_{0}(\h)$; note that $|\phi_\h^{(0)}| = |\phi_\h|$.  Our precisely stated main result is

\begin{theo} \label{supbound} Let  $(M^n,g)$ be a compact Riemannian manifold and $P_{u}(\h):C^{\infty}(M) \rightarrow C^{\infty}(M), \text{ with } u \in \bcal^k(\epsilon) \subset \R^k$ for $k \geq n$, be a family of admissible magnetic Schr\"odinger operators of the form (\ref{magnetic}). Then, for $\epsilon > 0 \text { and } |t_0| >0$ sufficiently small,
$$ \int_{\bcal^k(\epsilon)}  | \phi_{\h}^{(u)}(x) |^2 \, du \asymp 1$$ as $\h \rightarrow 0^+,$ uniformly for $x \in M.$
\end{theo}

\begin{rem}  Throughout, $f(\h) \asymp g(\h)$ will mean that there are constants $C_j >0$ for $j=1,2,$ such that $C_1 f(\h) \leq g(\h) \leq C_2 f(\h)$ for $\h \in (0,\h_0].$ The notation $f(\h) \gtrapprox g(\h)$ means that there is a constant $C >0$ such that $f(\h) \geq C g(\h)$ for $\h \in (0,\h_0].$ For any Lebesgue measurable ${\mathcal A} \subset \R^k$ we denote its measure by $|{\mathcal A}|$ and similarily, when $A \in GL(n;\R)$, we simply write $|A|$ for $|\det A|.$  Finally, $C>0$ denotes a constant which can vary from line to line. \end{rem}

 One consequence of Theorem \ref{supbound} is the following eigenfunction restriction bound:
 
% \begin{theo} \label{restriction}
 %Let $H \subset M$ be an arbitrary submanifold and $\phi_h \in C^{\infty}(M)$ be any $L^2$-normalized eigenfunction of $P_{0}(h)$ as above. Then, there is measurable subset $\mathcal A_h \subset \bcal(\delta)$ with $|\mathcal A_h| >0$ such that for $u \in \mathcal A_h,$
%$$ \int_{H}  | \phi_{h}^{(u)} |^2 \, d\sigma_H \asymp 1$$ as $h \rightarrow 0^+.$ In general, $A_h$ will depend on both $H$ and initial eigenfunction $\phi_h.$ \end{theo} 
  \begin{theo} \label{restriction}
 Let $H \subset M$ be a smooth  submanifold,  $(\phi_\h)_{\h \in (0,\h_0]}$ be  any family of $L^2$-normalized eigenfunctions of $P_{0}(\h)$  with $P_0(\h) \phi_{\h} = E(\h) \phi_{\h},$ and $E(\h) = E + o(1)$ as above.  Then, given any sequence $\Omega(\h) = o(1)$ as $h \rightarrow 0^+,$ there is a measurable subset $\mathcal A_\h \subset \bcal^k(\epsilon)$ with $ \lim_{\h \rightarrow 0^+} \frac{ |\mathcal A_\h| }{ |\bcal^k(\epsilon)|} = 1$ such that for $u \in \mathcal A_\h,$
$$  \int_{H}  | \phi_{\h}^{(u)} |^2 \, d\sigma_H \lessapprox \frac{1}{\Omega(\h)}$$ as $\h \rightarrow 0^+,$ with the implied constant depending on $H$.  In general, the set $\mathcal A_\h$ will depend on $H,$ the initial eigenfunctions $\phi_{\h}$, and $\Omega(\h)$. 
\end{theo} 
 
 Indeed, restricting the upper and lower bounds in Theorem \ref{supbound} to $x \in H,$ integrating over $H$,  and applying Fubini's Theorem to interchange the iterated integrals gives
 \begin{equation} \label{iterated}
\int_{\bcal^k(\epsilon)} \left( \int_{H} | \phi_\h^{(u)}|^2  \, d\sigma_H  \right)\, du \asymp 1. \end{equation}
 Theorem \ref{restriction}  then  follows from (\ref{iterated}) and the Chebyshev  inequality. Since $\Omega(\h) = o(1)$ is arbitrary, Theorem \ref{restriction} shows that for a generic deformed eigenfunction $\phi_\h^{(u)}$, its restriction bounds along any submanifold $H \subset M$ are much smaller than the deterministic universal upper bounds in \cite{BGT, Hu} and are essentially consistent with the ergodic case \cite{TZ1,TZ2}.
 
  Roughly speaking, Theorem \ref{supbound} and Theorem \ref{restriction} say that by  introducing a generic magnetic potential $\omega_u$ and deforming a Schr\"odinger eigenfunction $\phi_\h$ via the unitary magnetic propagator $e^{-it_0 P_{u}(\h)/\h},$  pointwise blowup  is destroyed on average as one varies over the intensity and orientation of the magnetic field perturbation.  This is related to the notion of ``fidelity" in nuclear magnetic resonance and has been revived more recently in quantum chaos \cite{C,P}.  The quantities studied are the states $e^{itP_u(\h)}e^{-itP_0(\h)}\psi_0$, where $\psi_0$ is some initial function or a family of functions depending on $\h$.  In the case where $\psi_{0,\h}$ are eigenfunctions with $P_0(\h) \psi_{0,\h} = E(\h) \psi_{0,\h},$ one has that  $e^{itP_u(\h)}e^{-itP_0(\h)}\psi_{0,\h}=e^{-it E(\h) /\h} e^{itP_u(\h)}\psi_{0,\h}$, which is a unitary multiple of the family we consider in Theorem \ref{supbound}.  The magnetic deformation is important here since it satisfies the non-degeneracy condition described in Definition \ref{admissible} which is crucial in Proposition \ref{keyprop}; at present, we do not know whether or not  our methods extend to the case of electric deformations of $P_0(\h)$ of the form $P_u(\h) = -\h^2 \Delta_g + V_u$, where $V_u \in C^{\infty}(M)$ for $u \in \mathcal{B}^k(\epsilon).$

  The paper is organized as follows: In Section \ref{motivation} we motivate our general results by discussing the case of a magnetic propagator $e^{it_0 P_u(\h)/\h}$ on $\R^n$  with $P_u(\h) = \langle \h D_x + u, \h D_x + u \rangle.$  Section \ref{FIO} gives background on the semiclassical microlocal analysis that is needed for the proof of our main result. Section 4 contains the proof of Theorem \ref{supbound} and Section \ref{ergodic} gives applications of Theorem \ref{supbound} to quantum ergodic sequences of eigenfunctions; in the latter section, we get asymptotic results (see Theorem \ref{ergodic case}). Finally, in Section \ref{examples}, two examples with extremal eigenfunction supremum bounds are given to illustrate the main result; these are the ground states of the one-dimensional  harmonic oscillator and  the zonal spherical harmonics in two dimensions.
  
  We thank Yaiza Canzani and Dmitry Jakobson for  detailed comments regarding earlier versions of the manuscript and for several stimulating discussions. We  are also grateful to St\'ephane Nonnenmacher for many useful comments/suggestions and for calling our attention to the physics literature on the related question of quantum fidelity. Finally, we thank the referee for detailed comments on improving the manuscript.

\begin{rem}  We note that the results here {\em do  not} give improvements for averaged eigenfunction $L^{\infty}$-bounds but do for restriction bounds on submanifolds.   Theorem \ref{supbound} gives $\sup_{x \in M} \int_{\bcal^k(\epsilon)} | \phi^{(u)}_\h (x)|^2 du={\mathcal O}(1),$ however this does not imply a similar bound for  \break $ \int_{\bcal^k(\epsilon)} \sup_{x \in M} | \phi^{(u)}_\h (x)|^2 du$, i.e. the reverse Fatou lemma does not hold in general.   Indeed, for both examples discussed in Section \ref{examples}, $\| \phi_{\h}^{(u)} \|^2_{L^{\infty}(M)}$ is still maximal even after averaging over the magnetic deformations. \end{rem}

 %(ii) \, Writing $P_\h(u) = P_u(\h)$, one can view the upper bound in  Theorem \ref{supbound} as an extremal Strichartz-type bound of the form
%$$ \| P_\h(\cdot) \|_{L^2_x \rightarrow  L^{\infty}_{x} L^2_u } = {\mathcal O}(1).$$ \smallskip
%The relevant evolution is in the $u$-variables with time $t_0$ fixed.

%(ii) \, Other applications of the averaging method to more general semiclassical propagators given by $\h$-Fourier integral operators, including spectral projectors  and FBI transforms, will be discussed elsewhere.

%The elliptic operator $P_{u}$ and the conformal Laplacian $\Delta_u$ are in general  {\em not} the same.  Consequently, the modes $\phi_{\lambda}^u$ and $\phi_{\lambda(u)}$ are in general also not the same. At present, we do not know how to get improved averaged sup bounds analogous to Theorem \ref{supbound} for the eigenfunctions, $\phi_{\lambda(u)},$ of the conformal Laplacian, $\Delta_u.$ 

%\end{rem}

\section{Motivation: Propagators given by quadratic forms in $\R^n.$}\label{motivation}

One of our main motivations here comes from the well-known, explicit Fourier multiplier formulas for non-degenerate quadratic propagators in  $\R^n$ (see \cite[Theorem 4.8]{Zw}).  Let $Q \in GL(n;\R)$ be a real-valued nonsingular $n \times n$ matrix and suppose $\{f_{\hbar}\}_{\hbar \in (0,h_0]}$ is a semiclassical family with $f_{\hbar} \in {\mathcal S}(\R^n)$ and $\| f_{\hbar} \|_{L^2} =1$ for all $\hbar \in (0,\hbar_0].$ Let $\langle Q D, D \rangle$ be the natural inner product where $D$ is the column vector $(D_{x_1},...,D_{x_n})$ and $D_{x_k} = \frac{1}{i} \partial_{x_k}.$  One has the explicit integral formula
\begin{equation} \label{motivation1}
e^{ \frac{i \hbar}{2} \langle Q D, D \rangle} f_{\hbar}(x) =  |\det Q|^{-\frac{1}{2}}  (2\pi \hbar)^{-\frac{n}{2}}   e^{i \frac{\pi}{4} sgn Q} \int_{\R^n} e^{- \frac{i }{2 \hbar} \langle Q^{-1} y, y \rangle } \, f_{\hbar}(x+y) \, dy. \end{equation}
We now consider slightly more general quadratic  polynomial expressions in the $D$'s; specifically, the case where $\langle Q D, D \rangle$ is replaced by the magnetic Schr\"{o}dinger operator $\langle \hbar D +u, \hbar D +u \rangle$ with constant magnetic potential $u \in \R^n.$  It is readily checked  that 
\begin{equation} \label{mag motivation}
f_{\hbar}^{(u)}(x):= e^{ \frac{i}{2 \hbar} \langle \hbar D + u, \hbar D + u \rangle} f_{\hbar}(x) =  (2\pi \hbar)^{-\frac{n}{2}} e^{\frac{i n \pi}{4}} e^{ \frac{i |u|^2}{2\h}}   \int_{\R^n} e^{-\frac{i}{2\hbar} |y-u|^2 } \, f_{\hbar}(x+y) \, dy. 
\end{equation}
%We note that the oscillatory integral in (\ref{motivation1} gets replaced by Laplace counterpart in (\ref{mag motivation}). This is simply due to the fact that the total symbol
%$\sigma ( \langle \hbar D + u, \hbar D + u \rangle) = |\xi|^2 + 2 i \langle u, \xi \rangle + |u|^2.$
Let $\chi \in C^{\infty}_{0}(\R^n)$ with $\chi(u) = 1$ when $|u| <1$ and $\chi(u) = 0 $ for $|u| >2.$  It follows from (\ref{mag motivation}) that for any $ x\in \R^n,$
\begin{equation}\label{mag motivation2} \begin{array}{ll}
\int_{\R^n} | f_{\hbar}^{(u)}(x)|^2 \, \chi(u) \, du  = (2\pi \hbar)^{-n} \int_{\R^{3n}} e^{\frac{i}{2\hbar} (  |y-u|^2 - |y'-u|^2 ) } \, f_{\hbar}(x+y) \, \overline{f_{\hbar}(x+y')} \, \chi(u) \, dy dy' du \\ \\
=(2\pi \hbar)^{-n} \int_{\R^{3n}} e^{\frac{i}{2\hbar} \langle y-y', y+y' - 2u \rangle } \, f_{\hbar}(x+y) \, \overline{f_{\hbar}(x+y')} \,\chi(u) dy dy' du.
\end{array}
\end{equation}
Finally, writing $g_{\hbar,x}(y):= f_{\hbar}(x+y)$ and making the change of variables
\begin{equation} \label{cov}
u \mapsto \xi(u;y,y') = \frac{y+y'}{2} - u,
\end{equation}
it follows that the last line in (\ref{mag motivation2}) can be rewritten in the form
\begin{equation} \label{mag motivation3}
(2\pi \hbar)^{-n} \int_{\R^{3n}} e^{\frac{ i \langle y-y',\xi \rangle}{\hbar}} \chi \left(  \frac{y+y'}{2} - \xi  \right) \, g_{\hbar,x}(y) \overline{g_{\hbar,x}(y')} \, dy dy' d\xi  = \langle Op^w_{\hbar}(a) g_{\hbar, x}, g_{\hbar, x} \rangle_{L^2(\R^n)},\end{equation}
with $a(y,\xi) = \chi(y - \xi) \in S^{0,0}(\R^{2n})$; see (\ref{symbol}) for more on symbol classes. By the Calderon-Vaillancourt Theorem (\ref{bdd}) and the fact that $\| g_{\hbar,x} \|_{L^2} = \| f_{\hbar} \|_{L^2} =1$ for any $x \in \R^n,$ it follows that for $\h \in (0,\h_0],$  %The idea of treating $u$ as a phase variable is paramount to the main results of this paper.  

\begin{equation} \label{motivationupshot}
\sup_{x \in \R^n} \int_{\R^n} | f_{\hbar}^{(u)}(x)|^2 \, \chi(u) \, du  \leq C_{n},\end{equation} with $C_n>0$  a purely dimensional constant.  To carry out the Kuranishi change of variables in (\ref{cov}), we have used that
\begin{equation} \label{guts}
\left| d_{u} \xi \right| = 2^{-n} \left| d_{u} d_{\xi} p_u  \right| \neq 0, 
\end{equation}
where, $p_u(x,\xi) = |\xi + u|^2$ is the Hamiltonian function. The non-degeneracy of the mixed Hessian on the RHS of (\ref{guts}) is central to the bound in (\ref{motivationupshot}) and motivates the notion of admissibility in Definition \ref{admissible}.  Our main result in Theorem \ref{supbound} extends the analysis above to a wide class of magnetic deformations on arbitrary compact manifolds. The reason we choose magnetic deformations is largely due to the implied non-degeneracy of the mixed Hessian in (\ref{guts}).

% which is explained in detail in Definition \ref{admissible}. 

%At present, we do not know whether the method here applies to other deformations involving electric potentials; the case of metric deformations will be discussed in \cite{CJT}. 

\section{Semiclassical pseudodifferential and Fourier integral operators} \label{FIO} 

%\cs include h-wf stuff and clearly describe sc families vs invidual eigenfunctions \cs

We briefly review the relevant calculus of semiclassical pseudodifferential and Fourier integral operators that will be used in the proof of Theorem \ref{supbound}. The reader is referred to \cite{DS, Zw}  for further details.
\subsection{Semiclassical pseudodifferential operators ($\h$-PsiDOs)} The basic semiclassical  symbol spaces are
\begin{eqnarray} \label{symbol}
\nonumber S^{m,k}_{cl}(T^*M): &=& \{ a \in C^{\infty}(T^*M \times (0,\h_0]); \\
&& a \sim_{\h \rightarrow 0} \h^{-m} \sum_{j=0}^{\infty} a_j(x,\xi) \h^{j} \text{ with } | \partial_{x}^{\alpha} \partial _{\xi}^{\beta} a_j(x,\xi)| \leq C_{\alpha,\beta} \langle \xi \rangle^{k-j-|\beta |} \}.  \end{eqnarray}
Here, we use the standard notation for  $\langle \xi \rangle : = \sqrt{ 1 + |\xi|^2}.$ The corresponding space of {\em{$\h$-pseudodifferential operators}} is 
\begin{equation} \label{psdo}
\Psi^{m,k}_{cl}:= \{ A_{\h}: C^{\infty}(M) \rightarrow C^{\infty}(M); A_{\h}= Op_{\h}(a) \, \text{with}\, a \in S^{m,k}_{cl}(T^*M) \}, \end{equation}

\noindent where the Schwartz kernels are locally of the form
$$Op_\h(a)(x,y) = (2\pi \h)^{-n} \int_{\R^n} e^{i \langle x-y,\xi \rangle/\h} a(x,\xi,\h) \, d\xi, \text{ for } a \in S^{m,k}_{cl}.$$
Given $A_{\h} \in \Psi^{m,k}_{cl},$ the principal symbol $\sigma(A_{\h}) := \h^{-m} a_0$ using the notation of (\ref{symbol}).
It is sometimes convenient to use the $\h$-Weyl quantization with corresponding Schwartz kernel
\begin{equation} \label{weyl}
Op_\h^w(a)(x,y) = (2\pi \h)^{-n} \int_{\R^n} e^{i \langle x-y,\xi \rangle/\h} a \left( \frac{x+y}{2},\xi,\h \right) \, d\xi, \text{ for } a \in S^{m,k}_{cl}. \end{equation}

One useful feature of the latter quantization scheme is that for $a(x,\xi,\h)$ real-valued, the corresponding Weyl quantization is formally self-adjoint with
$Op_\h^w(a)^* = Op_\h^w(a).$ The operators $A_\h \in \bigcup_{j,k \in \mathbb Z} \Psi^{j,k}_{cl}$ form an algebra; in particular,   given $A_{\h} \in \Psi^{m,k}_{cl} \text{ and } B_{\h} \in \Psi^{m',k'}_{cl},$ the composition
$ A_\h \circ B_\h \in \Psi^{m+m',k+k'}_{cl}$ has the principal symbol $ \sigma( A_\h \circ B_\h)(x;\xi) = \sigma(A_\h)(x,\xi) \cdot \sigma(B_\h)(x,\xi) = \h^{-m-m'} a_0(x,\xi) b_0(x,\xi)$ in terms of the local representation in (\ref{psdo}).

\subsubsection{Operator $L^2 \rightarrow L^2$ bounds} The Calderon-Vaillancourt Theorem \cite{DS, Zw} states that 
\begin{equation} \label{bdd}
A_\h \in \Psi^{0,0}_{cl}  \Longrightarrow \| A_{\h} \|_{L^2 \rightarrow L^2} \leq C_n \sup_{|\alpha|  + |\beta| \leq 2n +1} | \partial_x^{\alpha} \partial_{\xi}^{\beta} a(x,\xi;\h) |. \end{equation}
In addition, one has the weak Garding inequality:  For $A_\h \in \Psi^{0,0}_{cl}$ and $\h \in (0,h_0]$, with $\h_0>0$ sufficiently small,
\begin{equation} \label{garding}
\sigma(A_\h)(x,\xi) \geq \frac{1}{C} >0 \Longrightarrow A_{\h} \geq \frac{1}{2C} I \end{equation}
where $I$ is the identity operator. The results (\ref{bdd}) and (\ref{garding}) are used in the proof of Theorem \ref{supbound}.

The  family of eigenfunctions $(\phi_\h)_{\h \in (0,\h_0]}$, and the corresponding deformations $(\phi_\h^{(u)})_{\h \in (0,\h_0]}$, have compact $\h$-wave fronts supported in an arbitrarily small neighborhood of the characteristic variety  $ \{ (x,\xi) \in T^*M; |\xi|_g^2 + V(x) = E \}$ for $\h$ sufficiently small; the following section on semiclassical wavefronts provides more details.  Hence, we are only interested here in $A_\h \in \Psi^{m,-\infty}_{cl}$ since all symbols will have compact support in $\xi$ and therefore only the asymptotic behavior in $\h$ is relevant.

\subsubsection{Semiclassical wavefront sets $WF_{\h}$}
Let $( \phi_{\h})_{\h \in (0,\h_0]}$ be a family of tempered $L^2$ functions on $M$ in the sense that $\| \phi_{\h} \|_{L^2(M)} = {\mathcal O}(\h^{-N})$ for some $N>0$ as $\h \rightarrow 0^+.$  In the semiclassical case, one is interested in determining decay properties of the family $(\phi_{\h})_{\h \in (0,\h_0]}$, not just regularity properties as functions on $M$. This naturally leads one to  define the notation of a $semiclassical$ $wave$ $front$ $set$,  $WF_{\h} (\phi_{\h}),$ associated with the family of functions $(\phi_{\h})_{\h \in (0,\h_0]}$; the reader is referred to \cite[Chapter 8, Section 4]{Zw} for further details.  As in the homogeneous case, it is more natural to define the complement of the semiclassical wave front of the family $(\phi_{\h})_{\h \in (0,\h_0])}$ with
\begin{eqnarray} \label{wf}
\nonumber WF_{\h} (\phi_{\h})^{c} &=& \{ (x,\xi) \in T^*M; \exists a \in S^{0,0}_{cl} \,\, \text{with} \,\,  a(x,\xi) \neq 0 \,\, \\  && \text{and}  \,\, \| Op_{\h}(a) \phi_{\h} \|_{L^2} = {\mathcal O}(\h^{\infty}) \|\phi_{\h}\|_{L^2} \}. 
\end{eqnarray}

\subsubsection{Eigenfunction localization}
We will use the fact that eigenfunctions $\phi_\h$ and corresponding deformations $\phi_\h^{(u)}$ have compact $\h$-wave fronts to simplify expressions $e^{-it_0 P_{u}(\h)/\h} \phi_{\h} $  by $\hbar$-microlocaly cutting-off the propagators $e^{-it_{0} P_u(\h)/\h}$ near the energy level set $p_0^{-1}(E) = \{ (x,\xi) \in T^*M; |\xi|_g^2 + V(x) = E \}.$  Such procedures are standard (see \cite[Chapter 8]{Zw}) but, for the convenience of the reader, we briefly review the specific case at hand. By assumption $(P_0(\h) - E ) \phi_{\h} = o(1)$ and $P_0(\h) - E = -\h^2 \Delta_g + V - E$ is $\h$-elliptic off $p_0^{-1}(E).$ Let $\chi \in C^{\infty}_0(\R)$ be a standard cutoff function equal to $1$ near the origin and consider the associated cutoff with  
$$\chi^{(0)}_E(x,\xi): = \chi ( p_0 (x,\xi) - E ).$$   Set $\chi^{(0)}_E (\h) := Op_{\h}(\chi^{(0)}_E) \in \Psi_{cl}^{0,-\infty}(M)$ to be the corresponding $\h$-pseudodifferential cutoffs.  It now follows by a standard semiclassical parametrix construction, as in \cite[Theorem 6.4]{Zw}, that
$ \| \phi_{\h} - Op_{\h}(\chi^{(0)}_{E}) \phi_{\h} \|_{L^2} = {\mathcal O}(\h^{\infty})$
and so, $WF_{\h}( \phi_{\h}) \subset p_0^{-1}(E).$
Replacing $\phi_{\h}$ with $P_{0}^{k}(\hbar) \phi_{\h}$ above, it follows by a Sobolev lemma argument that for all $k \in \N,$
\begin{equation} \label{cut1}
\| \phi_{\h} - Op_{\h}(\chi^{(0)}_{E}) \phi_{\h} \|_{C^k} = {\mathcal O}_{k}(\h^{\infty}). \end{equation}

For $\phi_{\h}^{(u)},$ one has the analogous equation $ ( Q(\h) - E) \phi^{(u)}_{\h} = 0$ where, by the semiclassical Egorov theorem (\cite{Zw} Theorem 11.1), $Q(\h) = e^{-it_0 P_{u}(\h)/\h} P_0(\h) e^{i t_{0} P_u(\h)/\h} \in Op_{\h} (S^{0,2}_{cl}).$ Consequently, by a similar parametrix argument as above it follows that, with the cutoff 
$$\chi^{(u)}_{E}(x,\xi) = \chi ( p_0 ( G_u^{t_0}(x,\xi)) - E),$$
\begin{equation} \label{cut2}
\| \phi_{\h}^{(u)} - Op_{\h}(\chi^{(u)}_{E}) \phi_{\h}^{(u)} \|_{C^k} = {\mathcal O}_{k} (\h^{\infty}). \end{equation}
Thus, from (\ref{cut1}) and (\ref{cut2}),
\begin{equation} \label{cutoff upshot}
 \phi_{\h}^{(u)} = Op_{\h}(\chi^{(u)}_{E}) e^{-it_0 P_u(\h)/\h} Op_{\h}(\chi^{(0)}_{E}) \phi_{\h} + {\mathcal O}_{C^k}(\hbar^{\infty}). \end{equation}  We note the cutoff 
 $\chi^{(u)}_{E}(x,\xi)$ is defined using the magnetic flow $G_u^{t_0}:T^*M \rightarrow T^*M$, but $\chi^{(u)}_{E}(x,\xi) = \chi (p_0(x,\xi) - E + {\mathcal O}(|u|))$ and so, for $u \in \bcal^k(\epsilon),$  supp  $ (\chi_{E}^{(u)}) \subset \{ (x,\xi) \in T^*M; d((x,\xi), p_0^{-1}(E) ) \leq C \epsilon \}$ with appropriate $C>0.$ Here, $d:T^*M \times T^*M \rightarrow \R^+$ is the distance function relative to any Riemannian metric on $T^*M$. The main point is that, for $\epsilon >0$ small,  supp $(\chi^{(u)}_{E})$ remains localized near the hypersurface $p_0^{-1}(E)$ for all $u\in \mathcal{B}^k(\epsilon)$.

\subsection{Semiclassical Fourier integral operators ($\h$-FIOs)}
 Due to the compactness of $WF_\h( \phi_\h^{(u)}),$  it suffices here to consider operators  $F_\h$ with Schwartz kernels locally of the form
\begin{equation} \label{fiodefn}
F_\h(x,y) = (2\pi \h)^{-n} \int_{\R^n} e^{i \psi(x,y; \xi)/\h} a(x,y; \xi,\h) \, d\xi, \end{equation}
for $a(x,y,\xi, \h) \in C^{\infty}_0(U \times V \times \R^n \times (0, \h_0])$ with $a(x,y,\xi,\h) \sim_{\h \rightarrow 0} \sum_{j=0}^{\infty} a_j(x,y; \xi) \h^{j-m}$, $a_j(x,y;\xi) \in C^{\infty}_0(U \times V \times \R^n)$, for some $m \in \R$, with $U$ and $V$ being open subsets of $\R^n$.  We assume that $\psi$ is a non-degenerate phase in the sense of H\"ormander \cite{DS}, that is, $\{d_{x,y,\xi}(d_{\xi_1}\psi),...,d_{x,y,\xi}(d_{\xi_n}\psi)\}$ is linearly independent  on the critical set $C_{\psi}$, where
$$ C_{\psi}: = \{ (x,y,\xi) \in U \times V \times \R^n; d_{\xi} \psi(x,y,\xi) = 0 \}.$$  

We use this local formulation to define our operator between compact, $n$-dimensional manifolds $M$ and $N$.  In view of (\ref{fiodefn}), $F_\h$ can be associated to an immersed Lagrangian manifold $\Gamma \subset T^*M \times T^*N$, where
$$ \Gamma = \{ (x,d_x \psi; y, -d_y \psi); d_{\xi} \psi (x,y,\xi) = 0 \},$$ by pushing forward $C_\psi$ through the map $i_{\psi}: C_{\psi} \to \Gamma \subset T^*(M \times N)$ given by $(x, y, \xi) \to (x, d_x \psi, y, -
d_y \psi)$; the manifold $\Gamma$ then contains $WF_{\h}(F_\h).$  For $\h \in (0,\h_0]$ small, the operators $F_\h: C^{\infty}(M) \rightarrow  C^{\infty}(N)$, with Schwartz kernel locally of the form (\ref{fiodefn}), are called {\em $\h$-Fourier integral operators} and we write $F_\h \in I^{m,-\infty}_{cl} (M \times N; \Gamma).$

The symbol of a semiclassical
Fourier integral operator $F_\h \in I^{m,-\infty}_{cl}(M \times N, \Gamma)$, where $h \in (0,h_0]$, is a
smooth section of $\Omega_{1/2} \otimes L$, that is, the half-density bundle of $\Gamma$ twisted by the Maslov line bundle.  In terms of the semiclassical Fourier integral representation in (\ref{fiodefn}), it is
the square root $\sqrt{d_{C_{\psi}}}$ of the delta-function on
$C_{\psi}$ defined by $\delta(d_{\xi} \psi)$, transported to
its image in $T^* (M \times N)$ under $\iota_{\psi}$. Concretely, if $(\lambda_1, \dots,
\lambda_n)$ are any local coordinates on $C_{\psi}$, extended as
smooth functions in a neighborhood of $C_{\psi}$, then
$$ d_{C_{\psi}}: = \frac{|d \lambda|}{|\partial(\lambda,
d_{\xi}\psi)/\partial(x,y,\xi)|}, $$ where $d \lambda$ is the
Lebesgue density and $|\partial(\lambda,
 d_{\xi}\psi)/\partial(x,y,\xi)|$ denotes the determinant of the Jacobian matrix. In the semiclassical Fourier integral representation (\ref{fiodefn}), the symbol is
transported from the critical set  $C_{\psi}$
 by the Lagrangian immersion $i_{\psi}$ to the set $\Gamma$. At a point $(x_0, \xi_0, y_0, \eta_0) \in \Gamma$, the principal symbol of $F_\h$
is defined to be
$$\sigma(F_\h) (x_0, \xi_0, y_0, \eta_0) := \h^{-m} i^*_{\psi} a_0
\sqrt{d_{C_{\psi}}}. $$  This formulation gives a coordinate-invariant definition of the principal symbol.

For a more detailed treatment of the compactly-supported $\h$-Fourier integral operators in (\ref{fiodefn}), see \cite[Chapter 8]{GuSt} and \cite[Chapter 10]{Zw}.
\subsubsection{Magnetic propagators} 
Consider the magnetic propagators  $W_u(\h):C^{\infty}(M) \rightarrow C^{\infty}(M)$, where 
$$W_{u}(\h):= \exp ( - i t_0 P_u(\h)/\h)$$  with Schwartz kernel $ W_u(\h)(x,y)$ and $t_0 > 0$ small. The associated evolution operator
$W(\hbar): C^{\infty}(M) \rightarrow C^{\infty}(M \times \bcal^k(\epsilon))$ has Schwartz kernel
$$W(\hbar)((x,u),y); \,\,  (x,u) \in M \times \bcal^k(\epsilon), \, y \in M.$$ One clearly has
$$ W_u(\hbar)(\cdot,\cdot) = W(\hbar)((\cdot,u),\cdot ).$$ 
  A key point in our argument is to interchange the roles of $x \in M$ and $u \in \bcal^k(\epsilon),$ viewing $u \in \bcal^k(\epsilon)$  as {\em range} variables and $x \in M$ as parameters. In terms of the evolution operator $W(\hbar)$, we define
  $$\tilde{W}_{x}(\hbar)(\cdot,\cdot) :=  W(\hbar)( (x,\cdot), \cdot).$$  Therefore, one has the identity
  \begin{equation} \label{newkernel}
\tilde{W}_{x}(\hbar)(u,y) = W_{u}(\hbar)(x,y), \text{ for } (x,y,u) \in M \times M \times \bcal^k(\epsilon), \end{equation}
with corresponding operators 
$$ \tilde{W}_{x}(\h): C^{\infty}(M) \longrightarrow C^{\infty}( \bcal^k(\epsilon) ) $$
depending on the {\em spatial parameters} $x \in M.$ 

We define the family of cut-off propagators $W_{u,E}(\h):C^{\infty}(M) \rightarrow C^{\infty}(M)$, where
 \begin{equation} \label{cutoff prop}
 W_{u,E}(\h):= Op_{\h}(\chi^{(u)}_{E}) \, e^{-it_0 P_u(\h)/\h} \, Op_{\h}(\chi^{(0)}_{E}), \,\, \h \in (0,\h_0]. \end{equation}
 The  family of cut-off operators  $W_{u,E}(\h); \h \in (0, \h_0]$ are compactly-supported semiclassical Fourier integral operators, with 
 $$W_{u,E}(\h) \in I^{0,-\infty}_{cl}(M \times M; \Lambda_{t_0}),$$  where
\begin{eqnarray} \label{cutoff2}
\nonumber \Lambda_{t_0} &:=& \{ (x,\xi;y,\eta) \in T^*M \times T^*M; (x,\xi) = G_{u}^{t_0}(y,\eta), \, (x,\xi) \in \text{supp}(\chi_{E}^{(u)}), \\&& (y,\eta) \in \text{supp} (\chi_{E}^{(0)}) \}. \end{eqnarray}
 In analogy with (\ref{newkernel}), we define a family of cut-off operators $\tilde{W}_{x,E}(\h): C^{\infty}(M) \rightarrow C^{\infty}(\bcal^k(\epsilon))$ with Schwartz kernels
 \begin{equation} \label{newkernelcutoff}
 \tilde{W}_{x,E}(u,y)(\h):= W_{u,E}(x,y)(\h), \,\, (x,y,u) \in M \times M \times \bcal^k(\epsilon), \,\, \h \in (0,\h_0].\end{equation}  Hence, in view of the eigenfunction localization estimates in (\ref{cut1}) and (\ref{cut2}),
\begin{equation} \label{cutoff upshot 2} \begin{array}{ll}
 \int_{B^k(\epsilon)} | \phi_{\h}^{(u)}(x)|^2 du = \int_{B^k(\epsilon)} |  W_{u,E}(\h) \phi_{\h} (x)|^2 du + {\mathcal O}(\h^{\infty}) \\ \\
 =  \int_{B^k(\epsilon)} | \tilde{W}_{x,E}(\h) \phi_h (u) |^2 \, du + {\mathcal O}(\h^{\infty}).
\end{array} \end{equation}
Let $B(x_0,\epsilon) \subset M$ be a small geodesic ball centered at $x_0 \in M.$
Locally, for $x \in B(x_0,\epsilon) \subset M,$  it suffices to assume that $k=n$ (see (\ref{jac})). In the next section (see Proposition \ref{keyprop}), we prove that for $t_{0},\epsilon>0$ sufficiently small, the family of operators
 $$  \tilde{W}_{x,E}(\h) \in I^{0, -\infty}_{cl}( M \times \bcal^n(\epsilon); \Gamma_{x}); \, \, \h \in (0,\h_0], $$
 Moreover, under the admissibility assumption in Definition \ref{admissible},  $\Gamma_{x} \subset T^*M \times T^* \bcal^n(\epsilon)$ are
 canonical graphs, locally for $x \in B(x_0,\epsilon).$  It is at this point that we use the specific structure of the magnetic perturbations to show that (\ref{admissible}) is satisfied which, in turn, implies the graph condition on $\Gamma_x.$ Then, given the last line in (\ref{cutoff upshot 2}), we use a  $T^*T$-argument (also utilized in the proof of Corollary \ref{upshot}) with $T= \tilde{W}_{x,E}(\h)$, and the $L^2$-bounds (\ref{bdd}) and (\ref{garding}) locally uniform in $x\in B(x_0,\epsilon)$, to finish the proof of Theorem \ref{supbound}.
  
 %\cs blurb \cs
 
%  (see below) and this is used in the symbol computations prove   the upper and lower bounds in Theorem \ref{supbound} as well as the averaged asymptotic results in the ergodic case in section \ref{ergodic}.

\section{Proof of Theorem \ref{supbound}}
\subsection{Operator bounds for semiclassical propagators}

The proof of our main theorem is based on a well-known $L^2$ operator bound for non-degenerate $\h$-Fourier integral operators (see \cite[Chapter 10]{Zw}). 
We introduce some notation before proving an important proposition. Fix $x_0 \in M$ and with $0 < \epsilon < inj(M,g)$ small so that the geodesic ball $B(x_0,\epsilon)$ is centered at $x_0$ with radius $\epsilon >0.$  One can locally write $P_{u}(\h) = Op_{\h} ( \sigma_0(P_u) + \h i \sigma_{-1}(P_u) )$, where the principal symbol $$\sigma_0(P_u) = p_{u}(x,\xi) = \sum_{i,j}( \xi_i + \omega_i(x,u)) g^{ij}(x) ( \xi_j + \omega_j(x,u)) + V(x).$$
 In local coordinates, $$  \sigma_{-1}(P_u)(x,\xi) =  \frac{1}{\sqrt{|g|}} \sum_{i,j} \partial_{x_j} [ \sqrt{|g|} g^{ij}] \xi_i  +  \frac{1}{\sqrt{|g|}} \sum_{i,j} \partial_{x_j} [ \sqrt{|g|} g^{ij} \omega_i(x,u) ]. $$ 
 
For $x \in B(x_0,\epsilon)$ with $\epsilon >0$ sufficiently small, it follows by the admissibility assumption (\ref{admissible}) that for some subset $u' = (u_{i_1},..,u_{i_n}) \in \bcal^{n}(\epsilon),$ 
\begin{equation} \label{jac}
 \left| \frac{ \partial \omega(x,u',u'')}{ \partial u'} \right|  \geq C_{0}, \end{equation}
uniformly for $(x,u',u'') \in B(x_0,\epsilon) \times \bcal^{k}(\epsilon).$ Of course, the choice of $u'$ can vary depending on the geodesic ball $B(x_0,\epsilon).$  The proof of Theorem \ref{supbound} is local in $x$  and so, in view of (\ref{jac}), it suffices to  assume that $k=n,$ which also simplifies the statement below of  Proposition \ref{keyprop}. We now introduce some notation in preparation for the statement of Proposition \ref{keyprop}.
Given  local coordinate charts $U,V \Subset  \R^n,$  
%$\tilde{\phi}(y,u,\xi; x) \in C^{\infty}(V \times \bcal^{n}(\epsilon) \times \R^n \times U)$ to be
\begin{equation} \label{phase0}
\tilde{\phi}(y,u,\xi;x) := S(t_0,x,\xi;u) - \langle y, \xi \rangle, \end{equation}
where $(y,u,\xi;x) \in V \times \bcal^{n}(\epsilon) \times \R^n \times U$ and  $S(t,x,\xi;u)$  satisfies the initial-value problem  %\cs CHANGED $t_0$ IN HJ EQN FOR GENERATING FUNCTION S TO $t$ \cs
\begin{equation}
\begin{cases} \label{phase2}
\partial_t S(t,x,\xi;u) + p_u(x, \partial_{x}S(t,x,\xi)) = 0,  \\
 S(0,x,\xi;u) = \langle x,\xi \rangle.
\end{cases}
\end{equation}
The key technical step in the proof of Theorem \ref{supbound} is the following
%\cs should define the phase function before proof \cs
%\cs introduce local ball and choose the u' here. drop the $\iota$ stuff \cs
% Let $\pi_m: \mathbb{R}^n \rightarrow \mathbb{R}^m$ be the standard projection map $\pi_m(u_1,...u_m,u_{m+1},...,u_n) = (u_1,...,u_m).$  Then,  there is a  grouping of $u = (u',u'') \in \bcal^{k}(\epsilon)$ with $u' \in \bcal^{n}(\epsilon) :=  \pi_n(\bcal^{k}(\epsilon))$ and an associated inclusion map $\iota: u'  \mapsto (u',u''_0)$, for any arbitrarily chosen $u''_0 \in \bcal^{k-n}(\epsilon) :=  \pi_{k-n}(\bcal^{k}(\epsilon)),$ 

\begin{prop} \label{keyprop}  Assume that the family $(P_{u}(\h))_{\h \in (0,\h_0]}$ of magnetic Schr\"odinger operators satisfies (\ref{jac}) for  $ (x,u) \in \bcal(x_0,\epsilon) \times \bcal^n(\epsilon).$ Then,  for $\epsilon, |t_0| >0$ sufficiently small,

$$\tilde{W}_{x,E}(\h) \in I^{0,-\infty}_{cl}(M \times \bcal^n(\epsilon); \Gamma_{x}),$$ \smallskip
and the immersed Lagrangian
\begin{eqnarray*} 
\Gamma_{x} &=& \{ (u, d_{u} \tilde{\phi}(y,u,\eta;x) ; y, -d_y \tilde{\phi}(y,u,\eta;x)), \,\,
 d_{\eta} \tilde{\phi}(y,u,\eta;x) = 0 \} \subset  T^* \bcal^n(\epsilon) \times T^*M
\end{eqnarray*} is a canonical graph that is locally parametrized by the phase function  $\tilde{\phi}$ in (\ref{phase0}). 

Moreover, the local formula for the principal symbol of $\tilde{W}_{x,E}(\h)$  is
\begin{eqnarray}
\sigma( \tilde{W}_{x,E}(\h)) (u,\tau,y,\eta) =   \chi_{E}^{(u)}(x,\eta) \, e^{i \int_{0}^{t_0} \sigma_{-1}(P_u)( G^{t}_u(x,\eta) ) dt }  \\ 
\label{halfdensity} \times \frac{ | d_x \pi G^{t_0}_{u}(x,\eta)|^{1/2} }{ | d_{u} d_{\eta} \tilde{\phi}(y,u,\eta;x) |^{1/2} } \, | dy d\eta|^{\frac{1}{2}} \,\, 
\end{eqnarray}
where $\chi \in C^{\infty}_0(\R)$ and $( u,\tau,y,\eta) \in \Gamma_{x}.$ In (\ref{halfdensity})   $\pi:T^*M \rightarrow M$ is the canonical projection on the base manifold and we view $u= u(y,\eta)$ as a function of the local cotangent coordinates $(y,\eta) \in T^*M$ parametrizing $\Gamma_{x}.$

\end{prop}

\begin{proof}

 For $|t_0| < \epsilon_0$ small, the kernel
of $W_{u,E}(\h)$ can be locally written in the form (\cite[Chapter 10, Section 2]{Zw})
\begin{equation} \label{fio1}
W_{u,E}(\h)(x,y) = (2\pi \h)^{-n} \int_{\R^{n}} e^{i \phi(x,y,\xi;u) /\h} a(x,\xi;\h,u) \, d\xi + K_u(\h)(x,y) \end{equation}
where the amplitude $a(\cdot,u) \sim \sum_{j=0}^{\infty} a_j(\cdot,u) \h^j$ with $a_j(\cdot,u) \in  C^{\infty}( \mathcal{B}^n(\epsilon), C^{\infty}_0(U \times \R^n))$ with coordinate charts $U \Subset  \R^n$ and has the expansion discussed in the section on $\h$-FIOs.  In (\ref{fio1}), the operators $K_u(\h)$ are {\em residual} in the sense that
$$ | \partial_{x}^{\alpha} \partial_{y}^{\beta} K_u(\h)(x,y)| = {\mathcal O}_{\alpha,\beta,\epsilon}(\h^{\infty})$$ uniformly in $(x,y,u) \in U \times V \times \bcal^n(\epsilon)$, for $\epsilon >0$ small.
The phase function is of the form
\begin{equation} \label{phase1}
\phi(x,y,\xi;u) = S(t_0,x,\xi;u) - \langle y, \xi \rangle \end{equation}

Here,
\begin{equation} \label{magham}
p(x,\xi,u) = \sum_{i,j=1}^n g^{ij}(x) ( \xi_i + \omega_i(x,u)) ( \xi_j + \omega_j(x,u)) + V(x) \end{equation}
is the principal symbol of the operator $P_u(\h)$.  In view of (\ref{phase2}),  by a Taylor expansion of $S(t,x,\xi) $ around $t=0$, it follows that for $t=t_0$ small,
\begin{equation} \label{phase3}
S(t_0,x,\xi;u) = \langle x, \xi \rangle - t_0 p(x,\xi,u) + {\mathcal O}(t^2_0). \end{equation}
The remainder in (\ref{phase3}) is locally uniform in $(x,\xi,u) \in K \times \bcal^n(\epsilon)$ with $K \subset T^*M$ \ compact and $\epsilon >0$ small.  By implicit differentiation of (\ref{phase2}) in $(x,\xi)$, it follows that the same is true for derivatives, implying 
$$\| S(t_0,x,\xi;u) -  \langle x, \xi \rangle + t_0 p(x,\xi,u) \|_{C^{k}(K \times \bcal(\epsilon))} = {\mathcal O}_{k}(t^2_0).$$

The key step in the proof of Proposition \ref{keyprop} involves interchanging the roles of $x$ and $u.$ Consequently, we now view $u \in \bcal^n(\epsilon)$ no longer as parameters, but rather as range variables for the operators $\tilde{W}_{x,E}(\h): C^{\infty}(M) \rightarrow C^{\infty}(\bcal^n(\epsilon))$ in (\ref{newkernel}) which now depend on the local parameters $x \in B(x_0,\epsilon).$ 

To emphasize this point and following (\ref{phase0}), we write
\begin{equation} \label{newphase} \begin{array}{ll}
\tilde{\phi}(u,y,\xi;x) = S(t_0,x,\xi;u) - \langle y, \xi \rangle \\ \\
= \langle x - y, \xi \rangle - t_0 p(x,\xi,u) + {\mathcal O}(t^2_0) \\ \\
= \langle x - y, \xi \rangle - t_0 \sum_{i,j} g^{ij}(x) (\xi_i + \omega_i(x,u)) (\xi_j + \omega_j(x,u)) + {\mathcal O}(t^2_0),
\end{array} \end{equation}
where the last line follows from (\ref{phase3}).
Similarly, we define 
$\tilde{a}(u,\xi;x,\h) := a(t_0,x,\xi;u,\h).$  

From (\ref{fio1}), it follows that $\tilde{W}_{x,E}(\h) \in I^{0,-\infty}_{cl}(M \times \bcal^n(\epsilon);\Gamma_x)$. It remains to  show that $\Gamma_x$ is a canonical graph and to compute the principal symbol of $\tilde{W}_{x,E}(\h).$

 We recall that the admissibility assumption implies that locally, for $x \in B(x_0,\epsilon)$,  we have the invertibility  condition (\ref{jac})
\begin{equation} \label{jacobian}
| d_{u} (\omega_1(x,u),....,\omega_n(x,u)) | \geq C >0 \end{equation}
uniformly for $(x,u) \in B(x_0,\epsilon) \times \bcal^n(\epsilon)$ with $\epsilon >0$ small. 

We compute in geodesic normal coordinates  $x =(x_1,...,x_n)$ on $B(x_0,\epsilon)$ centered at $x_0$ (i.e. $x(x_0) = 0$); here onward, for convenience, we abuse notation somewhat by identifying points with the respective coordinate representations. Since $g^{ij}(x) = \delta_{j}^i + {\mathcal O}(|x|^2),$ a Taylor expansion in the $x$-variables in (\ref{magham}) around $x=0$ gives

\begin{equation} \label{taylorham} \begin{array}{ll}
p(x;\xi,u) = | \xi + \omega(x,u) |^{2} ( 1 + {\mathcal O}(|x|^2)) + V(x);
\end{array}
\end{equation}
in view of (\ref{newphase}),
\begin{equation}\label{newphase2}
\tilde{\phi}(u,y,\xi;x) = \langle x - y, \xi \rangle - t_0  ( \,   | \xi + \omega(x,u) |^{2} ( 1 + {\mathcal O}(|x|^2))  + V(x)  \, )  + {\mathcal O}(t^2_0). \end{equation}
From (\ref{newphase2}), it follows that  the mixed $(u,\xi)$-Hessian matrix of $\tilde{\phi}(u,y,\xi;x)$ has entries
\begin{equation} \label{taylorham2}
\partial_{u_i} \partial_{\xi_j}  \tilde{\phi}(u,y,\xi;x) = 2 t_0  (  \partial_{u_i} \omega_j(x,u) + {\mathcal O}(|x|^2) + {\mathcal O}(t_0) ) \text{ for } \,\, 1\leq i,j \leq n.\end{equation}
Consequently, from (\ref{jacobian}) and (\ref{taylorham2}), the operator $\tilde{W}_{x,E}(\h):C^{\infty}(M) \rightarrow C^{\infty}(\bcal^n(\epsilon))$ has Schwartz kernel

%\begin{equation} \label{graph1}
%| D_{u} D_{\xi} p(x,\xi,u) | \geq C >0 \end{equation}
%locally and uniformly for $(x,\xi,u) \in K \times B(\epsilon).$ In view of (\ref{phase3}) and (\ref{taylorham2}),

$$  \tilde{W}_{x,E}(\h)(u',y) = (2\pi \h)^{-n} \int_{\R^n}   e^{i \tilde{\phi}(u,y,\xi;x)/\h} \tilde{a}(u,\xi;x,\h)  \, d\xi + R(\h) \\ $$ 
\noindent with $R(\h)$ residual and for   $|t_0|,\epsilon >0$ sufficiently small,
\begin{equation} \label{graphcondition} 
| d_{u} d_{\xi} \tilde{\phi}(u,y,\xi;x)| = 2^n | t_{0} |^n \cdot  |  \, (  \, \partial_{u_i} \omega_j(x,u)  + {\mathcal O}(|x|^2)  + {\mathcal O}(t_0)  \, )_{1\leq i,j \leq n}  \, | 
\geq C |t_0|^n
\end{equation} locally and uniformly in $(u,y,\xi;x).$  From (\ref{graphcondition}) and  the Implicit Function Theorem, we conclude that the canonical relation of $\tilde{W}_{x,E}(\h): C^{\infty}(M) \rightarrow C^{\infty}(\bcal^n(\epsilon))$,  given by
\begin{eqnarray} \label{rel}
\nonumber \Gamma_{x}  &=& \{ (u,  d_{u} \tilde{\phi}; y, -d_y \tilde{\phi}) \in  T^*(\bcal^n(\epsilon)) \times T^*M; \,\, d_{\xi} \tilde{\phi} = 0 \} \\
&=& \{ (u,  d_{u} S; y, \xi) \in  T^*(\bcal^n(\epsilon)) \times T^*M; \,\, y = d_{\xi} S \} \end{eqnarray}  is a canonical graph for all $ x \in B(x_0,\epsilon)$ when $|t_0|, \epsilon >0$ are sufficiently  small.
From (\ref{rel}) it is clear that $(y,\xi)$ can be chosen as local parametrizing variables for $\Gamma_{x}.$

\begin{rem} We note that it is at this point that the specific form of the magnetic Hamiltionian $p(x,\xi,u)$  (\ref{taylorham}) is used to derive the required non-degeneracy of the phase function $\tilde{\phi}$ in (\ref{graphcondition}). \end{rem}

To determine the principal symbol $\sigma (\tilde{W}_{x,E}(\h))$, we compute the leading order term $a_0(u,\xi;x)$ in the amplitude of $ \tilde{W}_{x,E}(\h).$  Writing $\phi_t(x,y,\xi;u) =  S(t,x,\xi;u) - \langle y, \xi \rangle $, it follows that  $a_{0}(u,\xi;x) = a_{t_0}(x,\xi;u)$, where $a_t$ solves the initial-value problem for the first transport equation 
%$$ \partial_{t} a_t - H_{p_u} a_t + \left[ \frac{1}{2} \sum_{i,j} \frac{\partial^2 p_u}{\partial \xi_i \partial \xi_j }  \cdot \frac{ \partial^2 \phi}{\partial x_i  \partial x_j } - \sigma_{sub}  \right] a_t = 0$$
\begin{equation} \label{transport1} \begin{cases}
\partial_{t} a_t - H_{p_u} a_t  - \left[ \frac{1}{2} \sum_{i,j} \frac{\partial^2 p_u}{\partial \xi_i \partial \xi_j }  \cdot \frac{ \partial^2 S}{\partial x_i  \partial x_j } - i \sigma_{-1}(P_u)   \right] a_t = 0
 \\ \\
a_{0} = \chi^{(0)}_{E}  \mbox{ when } t=0 \end{cases}
\end{equation}
on the Lagrangian $\Gamma_{\phi_t} = \text{graph} ( G^{t}_{u}).$ 
  For $|t_0|$ small, a standard method of characteristics argument  gives the solution
\begin{equation}\label{symbol1}
a_{t_0}(x,\xi;u) = e^{i \int_{0}^{t_0} \sigma_{-1}(P_u)( G^{t}_u(x,\xi) ) dt } \cdot  | \det d_x  \pi \big( G^{t_0}_{u}(x,\xi) \big)|^{\frac{1}{2}} \cdot \chi^{(u)}_E(x,\xi), \end{equation}
where $G^{t_0}_{u}: T^*M \rightarrow T^*M$ is the time $t_0$ bicharacteristic flow generated by the magnetic Hamiltonian $p_u$ and $\chi^{(u)}_E(x,\xi) = \chi^{(0)}_{E} ( G^{t_0}_{u}(x,\xi)).$
%\begin{equation} \label{symbol2}
%a_{0} (u,\xi;x) = e^{i \int_{0}^{t_0} \sigma_{-1}(P_u)( G^{t}_u(x,\xi) ) dt } \cdot   | \det d_x  \pi G^{t_0}_{u}(x,\xi)|^{\frac{1}{2}} \cdot \chi_{E}^{(u)}(x,\xi).
%\end{equation}

The Leray density on $\Gamma_{x}$ is %\cs I DON'T SEE HOW THE $dy$ IS CANCELLING OUT BELOW AND I BELIEVE THERE SHOULD A JACOBIAN FACTOR APPEARING IN THE VERY LAST EQUALITY OF THE FIRST LINE; SINCE WE ARE WRITING IN INVARIANT FORM ANYWAYS, I THINK IT IS BETTER TO EXPLICITLY WRITE DOWN THE HOMOGENEOUS JACOBIAN FACTOR BUT LEAVE THE DENSITY TERM UNTOUCHED.  THE JACOBIAN FACTOR IS NONVANISHING AND BOUNDED, SO WE STILL GET THE LOWER BOUND RESULT.  EVERYTHING STILL WORKS OUT FINE EVEN WITHOUT REDUCING THINGS.\cs
$$ d_{C_{\tilde{\phi}} } = \frac{|du dy d\xi |}{ | d (d_{\xi} \tilde{\phi} )|} = \frac{|du dy d\xi |}{ | d (y - d_{\xi} S(u,\xi;x,t_0) )|} = |du d\xi|.$$
In terms of coordinates $(u,\tau = d_{u}S(u,\xi;x,t_0))$  that parametrize the canonical graph $\Gamma_{x}$,  we have %\cs CUTOFF $\chi_u$ IS MISSING BELOW AS WELL AS THE OSCILLATORY FACTOR\cs
\begin{equation} \label{leray} \begin{array}{ll}
\sigma (\tilde{W}_{x,E}(\h))(u,\tau) = a_{t_0}(u,\xi;x) \, |d_{u} d_{\xi} S |^{-\frac{1}{2}} \, |du d\tau|^{1/2} \end{array} \end{equation}
$$=e^{i \int_{0}^{t_0} \sigma_{-1}(P_u)( G^{t}_u(x,\xi) ) dt }  \,  \chi^{(u)}_{E}(x,\xi) \, \frac{ | d_x \pi G^{t_0}_{u} (x,\xi)|^{1/2} }{  |d_{u} d_{\xi} S |^{1/2} } \, |du d\tau|^{1/2} $$
$$= e^{i \int_{0}^{t_0} \sigma_{-1}(P_u)( G^{t}_u(x,\xi) ) dt } \, \chi^{(u)}_{E}(x,\xi)  \, \frac{ | d_x \pi G^{t_0}_{u} (x,\xi)|^{1/2} }{  |d_{u} d_{\xi} S |^{1/2} } \, |dy d\xi |^{1/2}$$
 for $(u,\tau;y,\xi) \in \Gamma_{x}. $  The fact $|dy d\xi| = |du d\tau|$, which follows from the map $(y,\xi) \mapsto (u,\tau)$ with $(u,\tau;y,\xi) \in \Gamma_{x}$ being a symplectomorphism, is used in the last line of (\ref{leray}).
We also note that in terms of geodesic normal coordinates in the ball $B(x_0,\epsilon),$ the Hessian matrix,
\begin{equation*}
(d _{u_i} d_{\xi_j} S)_{1 \leq i,j\leq n} = 2 t_0 (  \, \partial_{u_i} \omega_j(x,u) )_{1 \leq i,j \leq n} + {\mathcal O}(|x|^2) + {\mathcal O}(t_0) \, ),
\end{equation*} 
%$$ =  \text{diag}(\delta \omega_1(x),...,\delta \omega_n(x)) + {\mathcal O}(|u'|) + {\mathcal O}(|x|^2)$$ 
is non-degenerate for $(x,u) \in B(x_0,\epsilon) \times \bcal^n(\epsilon)$  and $ t_0, \epsilon >0$ small by the admissibility assumption on the family of magnetic potentials $(\omega_u)_{u \in \bcal^n(\epsilon)}$ in  Definition \ref{admissible}.  Hence, $ \tilde{W}_{x,E}(\h) \in  I^{0,-\infty}_{cl}(M \times \bcal^n(\epsilon); \Gamma_{x})$.

 \end{proof}

The next corollary is an immediate consequence of Proposition \ref{keyprop} and the semiclassical Egorov theorem.
 
\begin{cor} \label{upshot}
Given $(x,u) \in B(x_0,\epsilon) \times \bcal^n(\epsilon)$ with $\epsilon >0$ small and $\h \in (0,\h_0],$ the operators
\begin{equation} \label{egorov} 
A_{x, E}(\h):= [ \tilde{W}_{x,E}(\h) ]^* \circ  \tilde{W}_{x,E}(\h)  \in\Psi_{cl}^{0,-\infty}(M), 
\end{equation}
and have  principal symbol 
\begin{equation} \label{principal symbol} \sigma_0(A_{x,E}(\h))(y,\eta) =  |\chi_{E}^{(u(y,\eta))}(x,\eta) |^2 \, \frac{ |d_{x} \pi G_{(u(y,\eta))}^{t_0}(x,\eta) |} { |d_{u} d_{\eta} S|}. \end{equation}  
\end{cor}

Since $\Gamma_{x}$ is  a canonical graph, we can write the points in (\ref{principal symbol}) as $(u(y,\eta),\tau(y,\eta);y,\eta )\in \Gamma_{x}$  in terms of the $(y,\eta)$  parametrizing coordinates.

Moreover, it is clear from the stationary phase argument used in the proof of the $\h$-Egorov theorem that the family of $\h$-pseudodifferential  operators $A_{x, E}(\h)$  in Corollary \ref{upshot} depends regularly and locally uniformly on the $x$-parameters in the sense that the total symbol $\sigma(A_{x,E})(y,\eta;\h)$ satisfies the local estimates
\begin{equation} \label{localreg}
\| \sigma (A_{x, E}) \|_{C^k} = {\mathcal O}_k(1) \end{equation} uniformly for $x \in B(x_0,\epsilon)$ and  $\h \in (0,\h_0].$ 
\medskip

To finish the proof of Theorem \ref{supbound}, we note that in view of (\ref{deform1}) and Corollary \ref{upshot} with $x \in B(x_0,\epsilon)$,
\begin{equation} \label{mainproof} \begin{array}{lll}
\int_{\bcal^{n}(\epsilon)} | \phi_{\h}^{(u)}(x)|^{2} \,  du = \int_{\bcal^{n}(\epsilon)} | ( W_{u, E}(\h) \phi_\h  ) (x) |^{2} \, du  \\ \\
= \langle  \tilde{W}_{x, E}(\h) \phi_\h, \tilde{W}_{x, E}(\h) \phi_\h \rangle_{L^2(\bcal^n(\epsilon))} \\ \\
=  \langle A_{x, E}(\h) \phi_\h, \phi_\h \rangle_{L^2(M)}.
\end{array} \end{equation}  Since $A_{x, E}(\h) \in \Psi^{0,-\infty}_{cl}(M)$ satisfies (\ref{localreg}),  by  $L^2$-boundedness (see (\ref{bdd})),
\begin{equation} \label{mainprooflaststep}
\langle A_{x, E}(\h) \phi_\h, \phi_\h \rangle_{L^2(M)} \, \leq C''_{\epsilon} \| \phi_\h \|_{L^2(M)}^2 = {\mathcal O}_{\epsilon}(1) \end{equation} uniformly for $x \in B(x_0,\epsilon).$

To prove the lower bounds on $\int_{\bcal^n(\epsilon)} |\phi_\h^{(u)}(x)|^2 \, du,$ we note that for $|t_0|,\epsilon >0$ sufficiently small,%\cs ADDED A MISSING ``$|$" TO CLOSE OFF AN ABSOLUTE VALUE SIGN\cs
%\cs check exact form of cutoffs here \cs
$$ \frac{ |d_{x} \pi G^{t_0}_{u}(x,\xi) |} { |d_{u} d_{\xi} S|}  \,   |\chi_{E}^{(u(y,\eta))}(x,\eta) |^2  \,  _{|p_{0} = E}  \, \gtrapprox \, \frac{ | d_{x} \pi G^{t_0}_{u}(x,\xi) |} { |d_{u} d_{\xi} S|}  \, _{|p_{0}=E} \, \gtrapprox 1. $$  Thus, from the symbol formula in (\ref{principal symbol}), 
\begin{equation} \label{elliptic}
\sigma_0(A_{x, E}(\h))(x,\xi) \, _{| {p_{0} = E}} 
\gtrapprox 1.  \end{equation}
In view of (\ref{elliptic}),  the weak Garding inequality (\ref{garding}), and the mass concentration of the initial Schr\"odinger eigenfunctions $(\phi_{\h})_{\h \in (0,\h_0]}$ on $\{p_0 = E\}$ (\ref{cut1}), we have for all $\h \in (0, \h_0]$ sufficiently small,

\begin{equation} \label{lower}
\langle A_{x, E}(\h) \phi_\h, \phi_\h \rangle \gtrapprox 1 \end{equation}

\medskip

\noindent uniformly in $x \in B(x_0,\epsilon).$
Substitution of (\ref{lower}) into the last line of (\ref{mainproof}) gives the lower bound in Theorem \ref{supbound} uniformly for $x \in B(x_0,\epsilon).$

%the h-microlocal concentration of eigenfunction mass,
%$$ \langle Op_h( | \chi(p_u(x,\xi) - E |^2) \phi_h^{(u)}, \phi_h^{(u)} \rangle_{L^2(M)} = \langle \phi_h^{(u)}, \phi_h^{(u)} \rangle,$$

 Since $M$ is compact, we  select a cover of finitely many small balls $B_{j}(x_0,\epsilon)$ for  $j=0,...,N$ and repeat the above analysis above in each ball using the admissibility assumption (\ref{admissible}) to ensure that we can choose an $n$-tuple $(u_{i_1},...,u_{i_n})$ so that the non-degeneracy condition in (\ref{jac}) is satisfied in each ball, $B_j.$ This finishes the proof of Theorem \ref{supbound}.  \qed

%\cs I need to do a bit of revising in the next section, assume $k=n$ in the following  \cs

\section{The quantum ergodic case} \label{ergodic} 
%\cs ADDED MISSING ``$_0$'' SYMBOL TO $p^{-1}(\{E\})$  IN VARIOUS PLACES BELOW \cs
We assume here that $(M,g)$ is a smooth, compact $n$-manifold and  that the bicharacteristic flow of $p_0(x,\xi) = |\xi|^2_g + V(x)$ on $p_0^{-1}(E)$ given by
$$ G^t_0 = \exp t H_{p_0}: p_0^{-1}(E) \rightarrow p_0^{-1}(E)$$ is ergodic. As before, we consider a sequence $(\phi_{\h})_{\h \in (0,\h_0]}$  of $L^2$-normalized eigenfunctions of $P_0(\h) = -\h^2 \Delta_g^2 +V(x)$ with $P_0(\h) \phi_\h = E(\h) \phi_{\h}$ and $E(\h) = E + o(1)$ as $\h \rightarrow 0^+.$ We say that the sequence $(\phi_{\h})_{\h \in (0,\h_0]}$  is {\em quantum ergodic} if for any $B(\h) \in \Psi^{0,-\infty}_{cl}(M),$

  \begin{equation} \label{qe}
 \langle B(\h) \phi_\h, \phi_\h \rangle_{L^2} \sim_{\h \rightarrow 0} c_E \int_{p_0^{-1}(E)} \sigma_0(B(\h))(x,\xi) \, d\omega_E \end{equation}
 where $d\omega_E$ is Liouville measure and $c_{E}=  ( \text{vol}_{\omega_E}(p_{0}^{-1}(E)) )^{-1}.$  
 
 In the homogeneous case where $P_0(\h) = -\h^2 \Delta_g$ and $E=1,$ it is well-known that \cite{CV,Sch,Z2}   (\ref{qe}) is satisfied for a full asymptotic density of eigenfunctions.  The question of whether the asymptotic identity in (\ref{qe}) holds for {\em all} eigenfunctions is referred to as {\em quantum unique ergodicity} \cite{S2}.
 
% \cs up to here, reduce to k=n\cs

 Let $(\phi_{\h_{j_k}})_{k=1}^{\infty}$ be any subsequence of eigenfunctions with $\h_{j_k} \rightarrow 0 $ and $k \rightarrow \infty.$ One forms sequence of measures $d \mu_{\h_{j_k}}:= |\phi_{\h_{j_k}}(x)|^2 \, d vol (x)$ with the associated microlocal lifts $d\tilde{\mu}_{\h_{j_k}}$ defined by
 \begin{equation} \label{lift}
 d\tilde{\mu}_{\h_{j_k}}(a) := \langle Op_{\h_{j_k}}^{aw} (a) \phi_{\h_{j_k}}, \phi_{\h_{j_k}} \rangle_{L^2(M)}, \end{equation}
 where $a \in C^{\infty}(T^*M)$ and  $Op_{\h}^{aw}(a)$ denotes the non-negative semiclassical anti-Wick quantization \cite{Zw}. If  the measures $d\tilde{\mu}_{\h_{j_k}}$ converge in the weak-$*$ sense, one forms the associated weak-$*$ limit
 \begin{equation} \label{defect}
 d\tilde{\mu}: = w-\lim_{k \rightarrow \infty} d\tilde{\mu}_{\h_{j_k}}. \end{equation}

 In the quantum ergodic case, our main result in Theorem \ref{supbound} yields averaged pointwise asymptotics for the deformed eigenfunctions. We state the precise result here for future reference. In our case, we use the perturbed eigenfunctions $\phi_{\h}^{(u)}$ to form the microlocally lifted measures
 \begin{equation} \label{lift2}
 d\tilde{\mu}^{(u)}_{\h}(a):= \langle Op_{\h}^{aw}(a) \phi_{\h}^{(u)}, \phi_{\h}^{(u)} \rangle_{L^2(M)}. \end{equation}
 The following is local in $x \in B(x_0,\epsilon) \subset M$ and so, just as in Proposition \ref{keyprop},  it suffices to assume that $k=n$.

 \begin{theo} \label{ergodic case} Assume that $G^t_0:p_0^{-1}(E) \rightarrow p_0^{-1}(E)$ is ergodic and that $(\phi_\h)_{\h \in (0,\h_0]}$ is a quantum ergodic sequence of $L^2$-normalized eigenfunctions of the Schr\"{o}dinger operator $P_0(\h) = -\h^2 \Delta_g + V(x)$ with $P_0(\h) \phi_{\h} = E(\h) \phi_{\h}$ and $E(\h) = E + o(1).$ Then, under the same assumptions as in Theorem \ref{supbound} on the magnetic Schr\"{o}dinger operators $P_u(\h),$ it follows that
 
  (i) for all $u \in \bcal^n(\epsilon)$ sufficiently small, the sequence $(\phi_\h^{(u)})_{\h \in [0, \h_0]}$ is itself quantum ergodic up to a smooth density; that is,
 $$ w-\lim_{\h \rightarrow 0}  d \tilde{\mu}_{\h}^{(u)} =  \rho_u \, d\omega_{E,u},$$
 where $d\omega_{E,u} = \frac{dx d\xi}{d (p_u -E)} $ is Liouville measure on $p_u^{-1}(E)$  and we write $d\omega_E = d\omega_{E,0}$ for short.  Moreover, the density  $\rho_u$ is given by the Jacobian (i.e. the Radon-Nikodym derivative)
 %$\mu_u \in C^{\infty}(p_{u}^{-1}(E))$ is given by the Jacobian (i.e. the Radon-Nikodym derivative)
$$ \frac{  d (  ( G_{u}^{t_0} )^{*} \omega_{E} ) }{ d \omega_{E,u} }.  $$  
 (ii) For  $x \in B(x_0,\epsilon),$  one has the averaged pointwise asymptotics
 $$ \int_{\bcal^n(\epsilon)} |\phi_\h^{(u)}(x)|^{2} \, du \sim_{\h \rightarrow 0^+} c_E \, \omega_{\infty}(x),$$
 where $$ \omega_{\infty}(x) =   \int_{p_0^{-1}(E) }  \sigma_0(A_{x, E}(\h))   \, d\omega_E,$$
 with $\sigma_0(A_{x, E}(\h))$ given in Corollary \ref{upshot}.
 \end{theo}
 \begin{proof}
 The proof of (i) is an easy consequence of the semiclassical Egorov theorem. Indeed,
 for  $B(\h) \in \Psi^{0,-\infty}_{cl}(M)$ and any quantum ergodic sequence of $\phi_\h$'s, the $\h$-Egorov theorem gives
 \begin{equation} \label{qe1} \begin{array}{ll}
 \langle B(\h) \phi_h^{(u)}, \phi_\h^{(u)} \rangle_{L^2(M)} = \langle W_{u, E}(\h)^* B(\h) W_{u, E}(\h) \phi_\h, \phi_\h \rangle_{L^2(M)}  
 \\ \\ \sim_{\h \rightarrow 0} c_E \int_{p_{0}^{-1}(E)}  ( G_{u}^{t_0} )^* \sigma_0(B(\h)) d\omega_{E}.
 \end{array} \end{equation}
 In (\ref{qe1}), $G_u^{t}: T^*M \rightarrow T^*M$ denotes the time-$t$ bicharacteristic flow of the magnetic Hamiltonian $p_{u}.$  The flow-out $G_{u}^{t_0}: p_{0}^{-1}(E) \rightarrow G_{u}^{t_0} ( p_{0}^{-1}(E) )$ is  a diffeomorphism onto its image. Moreover, since $dp_{u} \neq 0$ on the set $\{p_u =E\},$ it follows  by Taylor expansion in $u$ and the Implicit Function Theorem that for $\epsilon >0$ small,   $p_{u}^{-1}(E)$ is diffeomorphic to $p_{0}^{-1}(E)$. Consequently, $p_u^{-1}(E)$ is diffeomorphic to $G_{u}^{t_0} ( p_{0}^{-1}(E) )$   and so, by (\ref{qe1}) and the change of variables formula,
 \begin{equation} \label{qe2} \begin{array}{ll}
 \langle B(\h) \phi_\h^{(u)}, \phi_\h^{(u)} \rangle_{L^2(M)} \sim_{\h \rightarrow 0^+} c_E \int_{ G_{u}^{t_0}( p_{0}^{-1}(E) ) }  \sigma_0(B(\h)) \, d (  ( G_{u}^{t_0} )^{*} \omega_E ) \\ \\
 =  c_E \int_{ p_{u}^{-1}(E) }  \sigma_0(B(\h)) \, \frac{ d (  ( G_{u}^{t_0} )^{*} \omega_E ) }{ d \omega_{E,u} } \, d\omega_{E,u}.\end{array} \end{equation}

  The proof of (ii)  follows from Theorem \ref{supbound} and (\ref{egorov}), since for the quantum ergodic sequence of eigenfunctions $(\phi_\h)_{\h \in [0, \h_0]}$ and {\em all} $x \in M,$
  $A_{x, E}(\h) \in \Psi^{0,-\infty}_{cl}(M)$ with uniform symbolic $C^k$-bounds. Consequently,
\begin{equation} \label{ergupshot} \begin{array}{ll}
\int_{\bcal(\epsilon)} |\phi_\h^{(u)}(x)|^{2} \, du  =  \langle A_{x, E}(h) \phi_\h, \phi_\h \rangle_{L^2(M)} \\ \\ \sim_{\h \rightarrow 0^+} c_E   \int_{p_0^{-1}(E) }   \sigma_0( A_{x, E}(\h)) \, d\omega_E.  \end{array} \end{equation}
%Here, $d\omega_{E}(u,x) = \frac{ d\xi}{ dp_u}$ is the Liouville measure on $\{ \xi; p_{u}(x,\xi)= E \}.$  
\end{proof}

The ergodic case will be discussed in more detail in \cite{CJT}.

%\section{Evolution in the $u$-variables}\label{evolution}
%\cs John: need to add a short section here \cs
%%\section{Averaged lower bounds for eigenfunction nodal volumes}
%In the section we give an application of Theorem \ref{supbound}

\section{Some Examples} \label{examples}

In this section, we explicitly compute the averaged pointwise bounds in the case of the 1-dimensional magnetic harmonic oscillator and the magnetic Laplacian on $S^2.$   Both examples have worst-case $L^{\infty}$-blowup for the unperturbed  eigenfunctions, $\phi_\h.$

%The former saturates the supremum bound (\ref{supbound}) while the latter exhibits interesting geometrical properties behind the averaging argument.

\subsection{Harmonic Oscillator}
Consider the 1-dimensional harmonic oscillator $P(x,\hbar D_x)=\frac{1}{2}(\hbar^2 D_x^2 + x^2)$ and the associated $L^2$-normalized ground state eigenfunction $\phi_{\hbar}(x)= (\pi\hbar)^{-\frac{1}{4}} e^{-\frac{x^2}{2\hbar}}$ with eigenvalue $\frac{\hbar}{2}$. This eigenfunction $\h$-microlocally concentrates near the critical point $p^{-1}(0) = \{ (0,0) \}.$  We show that the family of eigenfunctions $(\phi_\h^{(u)})_{\h \in (0, \h_0]}$  corresponding to the magnetic perturbations with
$$P_{u}(x,\h D_x) = P(x, \h D_{x}-u) = \frac{1}{2} [ ( \h D_x - u)^2 + x^2 ]$$
satisfies the bounds in Theorem \ref{supbound}.

Following \cite[page 129]{GS} and utilizing the formulas in (\ref{symbol1}), one easily derives the generalized Mehler formula for the perturbed propagator $e^{-itP_u(\hbar)/\hbar}$ with  $P_u(\hbar) = P(x,\hbar D_x-u).$ Explicitly,
\begin{equation} \label{master}
e^{-itP_u(\hbar)/\hbar}(\phi_{\hbar})(x) = \frac{1}{(2\pi \hbar)^{1/2}}\int_{\R} e^{i \Phi(t,x; \eta; u)/\hbar} a(t,x ;\eta ;u) \mathcal{F}_{\hbar}(\phi_{\hbar})(\eta) d \eta.
\end{equation}
Here, $\mathcal{F}_h(\phi_h)(\eta)$ denotes the semiclassical Fourier transform
$\frac{1}{(2 \pi \hbar)^{1/2}} \int_{\R} e^{-i y \eta/ \hbar} \phi_{\hbar}(y) dy$ for $\phi_h \in \mathcal{S}(\mathbb{R})$. The phase function is
\begin{equation} \label{harmonprop}
\Phi(t,x,\eta;u) = -\frac{ u^2  \sin t - 2 \eta u \sin t + (x^2  + \eta^2) \sin t  - 2  x u \cos t + 2x (u - \eta) }{2\cos t},
\end{equation}
and  $a(t,x,\eta;u) =  |\cos t|^{1/2}.$
% \cdot \chi^u(x;\eta)$, for $\chi \in C^{\infty}_0(\mathbb{R})$ and $\chi^u(x;\eta) = \chi(p_u(x;\eta))$. % \cs ADDED A COMMENT; FEEL FREE TO MODIFY OR ERASE IT IF YOU DON'T LIKE IT\cs For technical reasons regarding the method of steepest descent and analyticity in $\eta$, we will omit the cutoff function that appears in the symbol $a$.  The following computation continues to work even with the cutoff, but for the purposes of this computation this inclusion is unnecessary.
%the amplitude is
%\begin{equation} \label{harmonsymbol}
% a(t,x,\eta;u) = |\cos t|^{1/2} \cdot \chi \big(\frac{1}{2}((\eta-u)^2 + x^2)\big)
%\end{equation}
%where the  cutoff function $\chi \in C^{\infty}_0(\R)$ equal to 1 near 0.

%The phase function in (\ref{harmonprop}) is obtained by solving the eikonal equation in (\ref{phase2}) for the generating function  and the symbol $a$ in (\ref{harmonsymbol}) satisfies the usual transport equation with compact support in $(x; \eta).$  
When $u=0$ and $t \in (0, \pi)$ with $\hbar=1$, the operators in (\ref{master}) are elements of the unitary group generated by the standard harmonic oscillator  (see  \cite[p.129]{GS}).  Our aim is to compute 
\begin{equation} \label{oscillint}
I(\hbar,t,x) = \int_{-\epsilon}^{\epsilon} | e^{-itP_u(\hbar)/\hbar}(\phi_{\hbar})(x) |^2  \, du,
\end{equation}
where we fix $t=t_0 \in (0,\frac{\pi}{2})$ for the remainder of this computation.
Since ${\mathcal F}_{\h} (\phi_{\h})(\eta) =(\pi\hbar)^{-\frac{1}{4}} e^{-\frac{\eta^2}{2\hbar}},$ substitution in (\ref{master}) gives
\begin{equation} \label{master2}
e^{-it_0P_u(\hbar)/\hbar}(\phi_{\hbar})(x) = \frac{1}{(2\pi \hbar)^{3/4}}\int_{\R} e^{ \frac{i}{\h}  [ \Phi(t_0,x; \eta; u) + i \frac{\eta^2}{2} ]  } a(t_0,x,\eta ;u) d \eta.
\end{equation}
The  solutions $\eta_c = \eta_c(x;u,t_0) $ of the critical point equation $$d_{\eta} [ \Phi + i \frac{\eta^2}{2} ] = 0$$ are
$$ \eta_c(x;u,t_0) = i ( u \sin t_0 + x) e^{-it_0}.$$
%$$ (\tan t + i ) \eta_c = \frac{u \sin t - x}{\cos t}$$
Since $\Im  [ \Phi(t_0,x,\eta;u) + i \frac{\eta^2}{2} ] = \frac{\eta^2}{2} \geq 0,$ an application of steepest descent in (\ref{master2}) gives an asymptotic formula of the form

%First, we integrate in the $\eta$ variable.  Note that the semiclassical Fourier transform of our particular ground state equals $\frac{\sqrt{2}}{2} (\pi\hbar)^{-\frac{1}{4}} e^{-\frac{\eta^2}{4\hbar}}$.  The critical points of our phase function in $\eta$ are of the form $\eta = \frac{\sin(t_0)u + x}{\sin(t_0)}$; since $t_0 > 0$, it follows that these critical points are non-degenerate and  so an application of stationary phase gives

\begin{equation} \label{ho1} \begin{array}{ll}
  e^{-it_0P_u(\hbar)/\hbar}(\phi_{\hbar})(x)   \sim_{\h \rightarrow 0^+}  \hbar^{-\frac{1}{4}}  c(t_0) \, e^{  [ \Phi_2(x;t_0,u) + i \Phi_1(x;t_0,u) ]/\hbar} \end{array}
\end{equation}
where  $c(t_0)$ is a non-zero constant depending only on $t_0,$ $\Im \Phi_1 =0$, and
%\cs ADDED ADDITIONAL LINE TO CALCULATION OF THE PHASE \cs
\begin{eqnarray} \label{real phase}
 \nonumber \Phi_2(x;u,t_0) &=& - \Im \left( u  \eta_c  \tan t_0 - \frac{\eta_c^2}{2} \tan t_0 - \frac{ x \eta_c  }{\cos t_0} + i \frac{\eta_c^2}{2} \right) \\ 
 &=& -\Im \left(\frac{1}{2}(\tan t_0 + i)(u\sin t_0+x)^2 \right) = -\frac{1}{2} ( u \sin t_0 + x)^2.  
 \end{eqnarray}

%\begin{equation}
%\tilde{\Phi}(t,x;u) = -\frac{\sin(t)u^2 - 2(\sin(t)u + x) + x^2\sin(t) + \frac{\sin(t)u+x)^2}{\sin(t)} - %2\cos(t) x u - \frac{x(\sin(t)u + x)}{\sin(t)} + 2 x u}{-2 \cos(t)}.
%\end{equation}
There is an analogous formula for $\Phi_1,$ but it is irrelevant for our purposes.  We substitute the WKB formula (\ref{ho1}) into (\ref{oscillint}) and square the modulus of the integrand. Since $|e^{i \Phi_1 /\hbar}| =1,$ the oscillatory factor in (\ref{ho1}) drops out and one simply gets 
\begin{eqnarray} \label{laplace}
I(\hbar,t_0,x) \sim  \hbar^{-\frac{1}{2}} |c(t_0)|^2 \, \int_{-\epsilon}^{\epsilon}   e^{- \frac{ ( u \sin t_0 +  x )^2 }{2\hbar }} \, du.
\end{eqnarray}
%In this case $p_u^{-1}(\{0\}) = \{ (x,\xi) \in \R^2 : (\xi - u)^2  + x^2 =0\}$ and so the $\phi^u_{\hbar}(x)$ are locally concentrated near $x = 0.$  
By steepest descent,  the Laplace integral in (\ref{laplace}) gives an additional factor of $\hbar^{1/2}$ and it follows that 
\begin{equation} \label{sd}
I(\hbar,t_0,x) \sim  |c(t_0)|^2 |\sin t_0|^{-1} \asymp 1, \end{equation} locally and uniformly in $x \in \R$,
and (\ref{sd}) is consistent with Theorem \ref{supbound}.

\begin{rem} \label{big point}  Let $\ical \subset \R$ be any closed interval containing $[-\epsilon,\epsilon].$  From (\ref{ho1}) it is clear  that for each $u \in (-\epsilon,\epsilon),$ 
$$ \sup_{x \in \ical} | e^{-it_0P_u(\hbar)/\hbar} \phi_{\hbar}(x)| \sim \hbar^{-1/4}.$$
Consequently,
$$ \int_{-\epsilon}^{\epsilon} \sup_{x \in \ical} | e^{-it_0 P_u(\hbar)/\hbar} \phi_{\hbar}(x) |^2  du \sim {\hbar}^{-1/2}$$
and so the $L^{\infty}$-bounds of the deformed family of eigenfunctions $(\phi_{\hbar}^{(u)})_{\h \in (0, \h_0]}$ are  no  better than those of the ground state eigenfunctions $\phi_{\hbar}(x) = (\pi\hbar)^{-1/4} e^{-x^2/2\hbar}$. However, it is immediate from (\ref{sd}) that \begin{equation}
\sup_{x \in \ical} \int_{-\epsilon}^{\epsilon} | e^{-itP_u(\hbar)/\hbar} \phi_{\hbar}(x)|^2 du \asymp 1.                                                                                                                                                                                                                                                                                                                                                                                                                                          \end{equation}

\end{rem}

\subsection{Zonal harmonics on $S^2$}
We consider here the semiclassical Laplacian on the sphere $(S^2;round)$ with
\begin{equation}
P_0(\h) = -\h^2\Delta_{S^2} = -\h^2(\frac{1}{\sin^2(\phi)}\frac{\partial^2}{\partial \phi^2} + \frac{\partial^2}{\partial \theta^2}) - \hbar^2\cot(\phi) \frac{\partial}{\partial \phi}. 
\end{equation} 
where $(\theta,\phi) \in [0,2\pi] \times [0,\pi)$ denote usual spherical polar coordinates chosen so that  $(z_1,z_2,z_3) \in S^2$ is given  $z_1=\cos \theta \cos \phi, z_2 = \sin \theta \cos \phi, z_3 = \sin \phi.$  Then,  $x_1 = \phi \cos \theta, x_2 = \phi \sin \theta$ are the geodesic normal coordinates centered at the north pole $(0,0,1)$ and we choose $\h^{-1} \in \{ \sqrt{n(n+1)} \}_{n=1}^{\infty}.$ In the following, we continue to denote the geodesic normal coordinates centered at $(0,0,1) \in S^2$  by $x = (x_1,x_2).$

It is well-known that the $L^2$-normalized zonal spherical harmonic of degree $n$ is given by the formula 
\begin{equation} \label{intformula}
 Z_{n}(x) = P_{n}(\cos |x|) = 
(2 \pi n)^{1/2} \int_{-\pi}^{\pi} (\cos |x| + i \sin |x| \cos \tau)^{n} d \tau
\end{equation}
saturates the maximal H\"{o}rmander sup bound at $(0,0,1).$ Indeed it is obvious from (\ref{intformula}) that $|Z_{n}(0) | = (2 \pi)^{3/2}\sqrt{n}  \sim \h^{-1/2}. $
%where  the polar coordinates $(\phi,\tau) \in (-\pi,\pi) \times [0,2\pi)$ are given  by $z = \cos \phi, y=\sin \phi \cos \tau, x = \sin \phi \sin \tau,$ so that $\phi \in [0, \pi]$ is the angle on $S^2$ that traverses the vertical direction along the longitudinal lines, $P_n$ is the Legendre polynomial of degree $n$. From now on we put $\hbar = \frac{1}{n}$. 

  The method of steepest descent can  be applied in (\ref{intformula}) away from the north and south poles, corresponding to $|x| = \phi$ away from $0$ and $\pi$, and one easily gets that
  \begin{equation} \label{zonal3}
 \sup_{|x| \geq \epsilon} |Z_{\h}(x)| = {\mathcal O}_{\epsilon}(1) \end{equation}
 for $\h \in (0,\h_0].$ Hence, one can ignore this range where the eigenfunction is already uniformly-bounded prior to deformation.
Let $(Z_\h^{(u)})_{\h \in (0, \h_0]} := ((\exp -it_0 P_u(\h)/\h) Z_{\h})_{\h \in (0, \h_0]}$ be the deformed eigenfunctions where, for simplicity, we choose the (constant field) magnetic Schr\"odinger operators $P_u(\h): C^{\infty}(S^2) \rightarrow C^{\infty}(S^2)$ with
$$P_u(\h) = (\h d_x + i u)^* (\h d_x + iu); \,\, u = u_1 dx_1 + u_2 dx_2.$$ We verify directly that
\begin{equation} \label{zonal2}
\int_{\bcal^2(\epsilon)} | Z_\h^{(u)}(x)|^2 du \asymp 1  \end{equation}
uniformly for $x \in S^2.$ In particular, when $x$ corresponds to $(0,0,1)$, the worst case eigenfunction blow-up is destroyed by averaging over the magnetic field $u = u_1 dx_1 + u_2 dx_2 \in \Omega^1(S^2).$  In view of (\ref{zonal3}) it suffices to assume that $|x| = o(1)$ as $\h \rightarrow 0$ and we do so from now onward.
We split the ball $|x| =o(1)$ around the north pole into two pieces: Fix $\alpha \in (1/2,1)$ and first suppose  $|x| \gtrapprox \h^{\alpha}.$ In that case, an application of steepest descent in (\ref{intformula}) gives 
\begin{equation} \label{range1}
Z_\h(x)  \sim_{\h \rightarrow 0} \frac{ e^{i |x|/\h}  + e^{-i |x|/\h} }{ \sqrt{ \sin |x| } } \sim_{\h \rightarrow 0} \frac{ e^{i |x|/\h} + e^{-i |x|/\h} }{ \sqrt{ |x| } }. \end{equation}
In the complementary range where $|x| \lessapprox \h^{\alpha},$ we write the oscillatory integral (\ref{intformula}) as
\begin{equation} \label{range2} \begin{array}{lll}
Z_\h(x) = (2 \pi \h)^{-1/2} \int_{-\pi}^{\pi} \exp [ \h^{-1} \log (\cos |x| + i \sin |x| \cos \tau) d \tau \\ \\
\sim (2\pi \h)^{-1/2} \int_{-\pi}^{\pi} e^{i \sin |x| \cos \tau /\h} d\tau  \sim (2\pi \h)^{-1/2} \int_{-\pi}^{\pi} e^{i |x| \cos \tau /\h} d\tau. \end{array}
\end{equation}
The last line of (\ref{range2}) follows by Taylor expansion of the phase  function $\Phi(\tau;x) :=  \log (\cos |x| + i \sin |x| \cos \tau)$ around $x=0$ using that $\cos |x| = 1 + {\mathcal O}(|x|^2)$ and $ |x|^{-1} \sin |x|  = 1 + {\mathcal O}(|x|^2).$ The change of variables $\tau \mapsto \tau- x$ in (\ref{range2})  and application of the cosine law implies that  when $|x| \lessapprox \h^{\alpha},$
\begin{equation} \label{zonalupshot}
Z_{\h}(x) \sim (2\pi \h)^{-1/2} \int_{S^1} e^{i \langle x, \omega \rangle  / \h} \, d\omega. \end{equation}
Here we write $d\omega = d\tau$ for arc-length measure. In view of (\ref{range1}), it follows by stationary phase in $\omega \in S^1$ that the same asymptotic formula (\ref{zonalupshot}) is also valid when $|x| \gtrapprox \h^{\alpha}.$ 

Substitution of (\ref{zonalupshot}) and the general formula for the small-time propagator $e^{it_0 P_u(\h)/\h}$ gives 
\begin{equation} \label{computation1} \begin{array}{ll}
Z_\h^{(u)}(x) \sim (2\pi \h)^{-5/2} \int_{\R^2} \int_{S^1} \int_{|y| < \epsilon} e^{ i \h^{-1} [ \langle x -y, \xi \rangle - t_0 |\xi + u|^2 + {\mathcal O}(t_0^2) + {\mathcal O}(|x|^2)  +  \langle y, \omega \rangle ]  } \, a(x,\xi,t_0) \, dy d\omega d\xi \\ \\
=(2\pi \h)^{-5/2} \int_{\R^2} \int_{S^1} \int_{|y| < \epsilon} e^{ i \h^{-1} [ \langle x, \xi \rangle - t_0 |\xi + u|^2 + {\mathcal O}(t_0^2) + {\mathcal O}(|x|^2)  +  \langle y, \omega - \xi \rangle ]  } \, a(x,\xi,t_0) \, dy d\omega d\xi.
\end{array} \end{equation}
For $|x|, |t_0|$ sufficiently small, we apply stationary phase in $(y,\xi)$ in (\ref{computation1}) and get
\begin{equation} \label{computation2} \begin{array}{ll}
Z^{(u)}_\h(x) \sim (2\pi \h)^{-1/2} \int_{S^1} e^{i h^{-1} \Phi(x,\omega;t_0)}  \, a(x,\omega,t_0) \, d\omega, \end{array} \end{equation}
where $a(x,\omega,t_0) \asymp1$ and
\begin{equation} \label{phase}
 \Phi(x,\omega,t_0) = \langle x,\omega \rangle- t_0 |\omega +u |^2 + {\mathcal O}(t_0^2 +|x|^2). \end{equation}
From (\ref{computation2}), it follows that
\begin{eqnarray} \label{computation3} 
 \int_{\bcal^2(\epsilon)} && | Z^{(u)}_\h(x)|^2 \, du \sim \\
 && (2\pi \h)^{-1} \int_{\bcal^2(\epsilon)} \int_{S^1} \int_{S^1} e^{i \h^{-1}  [ \Phi(x,\omega,t_0) - \nonumber \Phi(x,\omega',t_0) ]}  \, a(x,\omega,t_0) a(x,\omega',t_0) \, d\omega d\omega' \, du_1 du_2. 
 \end{eqnarray}
 Writing $\omega = (\cos \tau, \sin \tau),$ the bound in (\ref{zonal2}) follows from (\ref{computation3}) and an application of  stationary phase in $(u_1,\tau)$ (or alternatively, $(u_2,\tau)$ depending on whether or not $\cos \tau = 0$ ).  This is consistent with our main result.
 
 \begin{rem} We note from (\ref{computation2}) that at the pole $x=0$ itself,
 $$ Z_\h^{(u)}(0) \sim (2\pi \h)^{-1/2} \int_{S^1} e^{i \h^{-1} [ -t_0 |\omega + u|^2 + {\mathcal O}_u(t_0^2) ]} a(0,\omega,t_0) \, d\omega $$
 $$ \sim (2\pi \h)^{-1/2} e^{i \h^{-1} t_0 (1 + |u|^2)} \, \int_{S^1} e^{-i \h^{-1}  t_0  [ \langle \omega, u \rangle + {\mathcal O}_u(t_0) ] } \, a(0,\omega,t_0) \,d\omega.$$  
 Thus, it follows by an application of stationary phase in the last integral that for $|t_0|  \ll \epsilon$ and $\h \in (0,\h_0(\epsilon)],$
 $$\sup_{|u| \geq \frac{\epsilon}{2} } | Z_\h^{(u)}(0)|  = {\mathcal O}_{\epsilon}(1).$$  
 By a similar argument, one can show that for $\epsilon >0$ small and  with large enough constant $C_0 >0,$ 
 \begin{equation} \label{locunif}
 \sup_{|x| \leq \frac{\epsilon}{C_0}} \sup_{|u| \geq \frac{\epsilon}{2}}   | Z_{\h}^{(u)}(x)|  = {\mathcal O}_{\epsilon}(1). \end{equation}

%Using a covering argument and the locally-uniform bounds (\ref{locunif}), one can show that for any smooth curve $H \subset S^2$, the perturbed eigenfunctions $Z_{\h}^{(u)}$  satisfy the restriction bounds
%$ \int_{H} |Z_{\h}^{(u)}|^2 d\sigma_H \asymp 1$ {\em for all} $|u| \geq \epsilon$ and $\h \in (0,\h_0(\epsilon)],$ improving on the Chebyshev argument in Theorem \ref{restriction}. 
\noindent At present, we do not know whether such deterministic local supremum bounds for the family $(\phi_\h^{(u)})_{\h \in (0, \h_0]}$  hold in greater generality.
 
\end{rem}

\end{document}